\documentclass[11pt]{article}
\usepackage{graphicx} 
\usepackage{amsmath,amsthm,amssymb}                
\usepackage{graphicx, color}
\usepackage{amssymb}
\usepackage{epstopdf}
\usepackage{pdfsync}
\usepackage{makeidx}

\usepackage{geometry} 
\newtheorem{definition}{Definition}[section]
\newtheorem{lemma}[definition]{Lemma}
\newtheorem{theorem}[definition]{Theorem}
\newtheorem{proposition}[definition]{Proposition}

\newtheorem{remark}[definition]{Remark}

\def\e{\varepsilon}
\newcommand{\weak}{\rightharpoonup}

\title{Dimension reduction of fractional Sobolev seminorms\\ on thin domains}
\author{Andrea Braides\footnote{Department of Mathematics, University of Rome Tor Vergata, via della ricerca scientifica 1, 00133 Rome, Italy} , Andrea Pinamonti\footnote{Department of Mathematics, University of Trento, via Sommarive 14, 38123 Povo, Italy} , and Margherita Solci\footnote{DADU, Universit\`a di Sassari, piazza Duomo 6, 07041 Alghero, Italy}}
\date{}

\begin{document}
\maketitle

\begin{abstract}
We study the asymptotic behaviour of  Gagliardo seminorms in $H^s$ defined on thin films $\Omega_\e=\omega\times(0,\e)$. The first relevant order is $\e^{1-2s}$, at which the corresponding limit captures the vertical fractional oscillations through one-dimensional sections. The second relevant order produces dimension-reduction regimes that undergo a qualitative transition at the critical exponent $s=\tfrac12$. For $s<\tfrac12$, the dominant contribution is driven by interactions at finite planar distance, and the dimension-reduction scale is $\e^2$. In this regime, the limit is a lower-dimensional \emph{fractional} energy with an effective gain of $\tfrac12$ in the differentiability index. At the critical exponent $s=1/2$, the dimension-reduction scale is $\e^{2}|\log\e|$, and the limit is {\em local}, with dominant interactions at scales between $\e$ and $1$, producing a Dirichlet-type limit on $\omega$. For $s>\tfrac12$, the dominant contribution is instead driven by interactions at distances of order $\varepsilon$, the dimension-reduction scale is $\e^{3-2s}$, and the second-order $\Gamma$-limit is still local. We also study the case $s=s_\e\to 1^-$, showing a Bourgain--Brezis--Mironescu-type result.

{\bf Keywords:}
fractional Sobolev spaces,
dimension reduction,
thin films,
nonlocal energies,
Gagliardo seminorm,
Bourgain--Brezis--Mironescu limit.

{\bf MSC codes:}
49J45, 46E35, 35R11, 49J53, 74K35
\end{abstract}

\section{Introduction}
Dimension-reduction problems arise naturally in the study of variational models defined on thin structures. In many physical situations one considers a family of domains whose thickness tends to zero and investigates the corresponding asymptotic behaviour of the associated energy functionals. Classical examples include thin films, membranes, and plates in nonlinear elasticity, where three-dimensional models give rise to effective lower-dimensional theories. The rigorous derivation of such reduced models has been extensively studied using $\Gamma$-convergence methods, starting from the seminal work of Le Dret and Raoult \cite{LDR} and later developments summarized e.g.~in the references \cite{GCB,Handbook,BFF,MR2210909,MR1916989} .

Most results in this area concern local energies, typically involving gradients. The simplest example is the Dirichlet integral
\begin{equation}\label{intro-1}
\int_{\Omega_\e} |\nabla u|^2dx.
\end{equation}
defined on thin domains
$$
\Omega_\e=\omega\times(0,\e)\subset \mathbb R^d,
$$ where $\omega$ is a (bounded) open set in $\mathbb R^{d-1}$ and $\e$ represents the thickness of the film. In this classical setting the asymptotic analysis reveals a separation between derivatives along the thin direction and derivatives along the planar variables, leading to effective lower-dimensional energies defined on $\omega$. This is done through an asymptotic analysis of functionals \eqref{intro-1} as $\e\to 0$. In order to provide a common analytical framework, it is customary to scale the variable $u$ to a function $v\in H^1(\omega\times (0,1))$ as follows
\begin{equation}\label{intro-2}
v(x',x_d)=u(x',\e x_d), \qquad x'\in\omega,\  x_d\in(0,1),\ \varepsilon>0.
\end{equation}
Here the prime indicates quantities depending on (or operations with respect to) the variable $x'=(x_1,\ldots,x_d)$. This change of variables allows one to view functionals \eqref{intro-1} as defined on $H^1(\omega\times (0,1))$. Properly scaling the energies ({\em dimension-reduction scaling}; for the functional \eqref{intro-1} simply dividing by $\e$), it is shown that their limits are actually finite on functions depending only on $x'$, thus understood as functions in $H^1(\omega)$.

In contrast, much less is known when the energy is nonlocal. Nonlocal variational energies arise naturally in many contexts, including anomalous diffusion, peridynamics, and fractional phase-transition models. In particular, fractional Sobolev seminorms play a fundamental role in the theory of fractional Laplacians and nonlocal operators. A comprehensive overview of fractional Sobolev spaces can be found in the survey of Di Nezza, Palatucci, and Valdinoci \cite{DPV}. A key structural result is the celebrated formula of Bourgain, Brezis, and Mironescu \cite{BBM}, which shows that fractional seminorms converge to the Dirichlet energy as the fractional exponent tends to one. This paper had a major impact and it was later generalized in many directions (see e.g. \cite{ponce,MR1942116,MR4275122,MR3556344,MR4682567, MR5012373,Davo, MR4999552,MR4788002}) and to different functional spaces (see e.g. \cite{MR4525722, MR1942130,MR2832587,MR3132740, ponce}).

Despite the extensive literature on fractional Sobolev spaces and nonlocal energies, the interaction between nonlocal kernels and geometric constraints remains comparatively unexplored. In particular, the asymptotic behaviour of fractional energies on thin domains presents new difficulties. Unlike the classical Sobolev case, fractional seminorms couple interactions across all length scales, and therefore do not naturally decouple into derivatives along different directions. As a consequence, the derivation of effective lower-dimensional models requires a more delicate analysis.

The aim of this paper is to study the asymptotic behaviour of fractional Sobolev seminorms on thin domains, of the form
\begin{equation}\label{intro-5}
\lfloor u\rfloor^2_{s}(\Omega_{\epsilon})=\int_{\Omega_\e\times\Omega_\e}\frac{(u(x)-u(y))^2}{|x-y|^{d+2s}}dx dy,
\end{equation}
with $s\in(0,1)$. Our goal is to identify the relevant scaling regimes and to compute the
corresponding $\Gamma$-limits of the associated energies. A key feature of the analysis is that the asymptotic behaviour depends crucially on the integrability of the interaction kernel, which is governed by the fractional exponent.

Our results reveal the existence of two qualitatively different dimension-reduction
re\-gimes separated by the critical exponent $s = 1/2$. Before analyzing the dimension-reduction
 process, we preliminarily highlight the first relevant scaling, which
is of order $\e^{1-2s}$. At this scaling, the dominant contribution to the energy arises from oscillations along the thin direction. The corresponding $\Gamma$-limit captures the fractional oscillations of the rescaled functions along one-dimensional vertical sections. In particular, the limit energy can be expressed as the integral over $\omega$ of one-dimensional fractional seminorms.

At the next scaling order, the behaviour depends on the value of $s$. When $s<1/2$, long-
range interactions dominate and the resulting effective energy remains nonlocal. In this
regime the limit functional corresponds to a fractional seminorm defined on the planar
domain $\omega$, with an {\em effective gain of $1/2$ in the differentiability index}. When $s>1/2$
instead, the dominant interactions occur at small scales of order $\e$, and the effective
energy becomes local. In this case the $\Gamma$-limit reduces to a Dirichlet-type functional on
$\omega$. The exponent $s=1/2$ represents a critical threshold at which the two mechanisms
balance and a logarithmic correction appears in the scaling. 

These results show that the dimension-reduction behaviour of fractional energies is
governed by the competition between the thickness of the domain and the
characteristic interaction length scale of the kernel. In particular, the analysis reveals a
transition from nonlocal effective energies to local ones depending on the integrability
of the kernel.

Finally, we also consider the regime in which the fractional exponent approaches one
simultaneously with the thickness parameter. In this case, we recover, through a suitable
renormalization, a Bourgain--Brezis--Mironescu-type limit that connects the present
analysis with the classical thin-film asymptotics for Dirichlet energies.

The results are complemented by compactness results, which suggest the relevant dimension-reduction convergence of the scaled functions $v$ for the computation of the $\Gamma$-limits.
We briefly summarize (and slightly simplify for the sake of exposition) the results below:

\smallskip
\noindent{\bf Vertical slicing (Theorem \ref{fulldim})}
{\em Let $s\in(0, 1)$, and define 
$$
E^s_\e(u)=\frac1{\e^{1-2s}}\lfloor u\rfloor^2_{s}(\Omega_\e).
$$
Then we have 
\begin{equation}
    \Gamma\hbox{-}\lim_{\e\to 0} E^s_\e(v)=C_{s,d}\int_{\omega}\int_0^1\int_0^1\frac{|v(x^\prime, x_d)-v(x^\prime, y_d)|^2}{|x_d-y_d|^{1+2s}}\, dx_d\, dy_d\, dx^\prime.
  \end{equation}
The $\Gamma$-limit is taken with respect to the weak convergence in $L^2(\omega \times (0,1))$ of the rescaled functions $v$.}

This result, in particular, shows that the limit energy describes the fractional oscillations of the function along one-dimensional vertical sections.
Remarkably, the explicit constant $C_{s,d}$,
arising from the change of variables that allows for a Fubini-type argument, can be written in terms of a \emph{hypergeometric} function, which itself can be expressed in terms of Euler's Gamma function.
The result can also be extended to varying $s_\e$, taking into account the one-dimensional version of the Bourgain-Brezis-Mironescu result if $s_\e\to 1$.

At the next order of the asymptotic expansion, the behaviour changes depending on the
value of the fractional exponent $s$. The following theorem characterizes the resulting
dimension-reduction regimes.

\smallskip
\noindent{\bf Dimension Reduction (Theorems  
\ref{drconv}, \ref{Gammaths1}, and \ref{Gammaths0})}

\smallskip
{\em 
{\bf 1. Low-integrability regimes.} Let $s<1/2$. At scale $\e^2$ we have
 \begin{equation}\label{drT-1}
    \Gamma\hbox{-}\lim_{\e\to 0} \frac{1}{\e^{2}}\lfloor u\rfloor^2_{s}(\Omega_\e)=\int_\omega\int_\omega 
\frac{|v(x')-v(y')|^2}{|x'-y'|^{d+2s}}\, dx^\prime  dy^\prime=\lfloor v\rfloor^2_{s+\frac12}(\omega).
  \end{equation}
  This limit is calculated with respect to the weak topology in $L^2(\omega\times(0,1))$.
  
\smallskip  {\bf 2. Critical regime.} For $s\ge1/2$  we have an improved dimension-reduction coerciveness result in the strong topology of $L^1_{\rm loc}(\omega\times(0,1))$. With respect to this convergence, for $s=1/2$ the correct scaling is $\e^2|\log\e|$, and we have
\begin{equation}\label{drT-2}
    \Gamma\hbox{-}\lim_{\e\to 0} \frac{1}{\e^{2}|\log\e|}\lfloor u\rfloor^2_{1/2}(\Omega_\e)=\frac{\sigma_{d-1}}{2(d-1)}\int_\omega
|\nabla' v|^2dx'.
  \end{equation}
}

\smallskip
{\em 
{\bf 3. High-integrability regimes.} Let $s>1/2$. At scale $\e^{3-2s}$
we have
\begin{equation}\label{drT-3}
    \Gamma\hbox{-}\lim_{\e\to 0} \frac{1}{\e^{3-2s}}\lfloor u\rfloor^2_{s}(\Omega_\e)=\frac{1}{1-s} K_{s,d}\int_{\omega} |\nabla' v|^2 dx'.
  \end{equation} }
  
   The coefficient $K_{s,d}$
in \eqref{drT-3} can also be expressed as a hypergeometric integral, diverging as $s\to1/2+$. Note that by the Bourgain-Brezis-Mironescu result, also the functional on the right-hand side of \eqref{drT-1} diverges as $s\to1/2^-$.

\smallskip
The dimension-reduction regimes can also be analyzed for varying $s_\e$. In particular for $s_\e\to 1^-$ we recover an analog of the Bourgain-Brezis-Mironescu result, which implies the following {\em separation of scale} phenomenon at both the first and the second scaling.

\medskip
\noindent{\bf Bourgain--Brezis-Mironescu-type Expansion (Theorem \ref{gammaexp}).}
{\em
Let $\e^{1-s_\e}\to 1$; that is, let
$$
1-s_\e<\!<\frac1{|\log\e|}.
$$
Then, we have
$$
(1-s_\e)\lfloor u\rfloor^2_{s_\e}(\Omega_\e)\ \buildrel\Gamma\over =\ \frac1\e\frac{
\sigma_{d}}{2d}\int_{\omega}\int_0^1\Big|\frac{\partial v}{\partial x_d}\Big|^2 dx_d\, dx^\prime
+\e\frac{\sigma_{d}}{2d}\int_{\omega} |\nabla' v|^2 dx'
+o(\e)
$$
in the sense of $\Gamma$-expansions} \cite{BT}.

\smallskip
These results show that the dimension-reduction behaviour of fractional Sobolev
energies depends crucially on the integrability of the kernel. In particular, the critical exponent $s=1/2$ separates regimes in which the effective lower-dimensional energy is nonlocal from those in which it becomes local. To the best of our knowledge, this is the first systematic asymptotic analysis describing the dimension reduction of fractional Sobolev energies in thin domains, including the critical transition at $s=1/2$.

\bigskip
The paper is organized as follows. In Section 2 we introduce the notation and recall basic facts on fractional Sobolev spaces and dimension-reduction convergence. Section 3 is devoted to the analysis of the first scaling regime and the corresponding $\Gamma$-limit in
the thin direction. In Section 4 we study the dimension-reduction regimes, identify the relevant scaling laws depending on $s$, and compute the associated $\Gamma$-limits. Finally, in Section 5 we reformulate the results in terms of $\Gamma$-expansions, highlighting the analogy with the classical asymptotic expansions for local thin-film energies.

\bigskip
\noindent{\bf Acknowledgments.} 
This work has been inspired by a suggestion of Fran\c cois Murat.  A.B. is partially supported by the MIUR Excellence Department Project 2023-2027 MatMod@TOV awarded to the Department of Mathematics, University of Rome Tor Vergata. A.P. is supported by the University of Trento, the MIUR-PRIN 2022 Project \emph{Regularity problems in sub-Riemannian structures}  Project code: 2022F4F2LH and the INdAM-GNAMPA 2025 Project \emph{Structure of sub-Riemannian hypersurfaces in Heisenberg groups}, CUP ES324001950001. M.S. is partially supported by the Department of Architecture, Design and Planning of the University of Sassari in the framework of the projects {\em DM737/2021 and DM737/2022-23}. 
The authors are members of GNAMPA, INdAM.

\section{Notation and preliminaries}\label{notation}

If $x\in\mathbb R^d$, then we write $x'=(x_1,\ldots, x_{d-1})$, and use the notation $x=(x',x_d)$.
We also write 
$$
\nabla'u=\Big(\frac{\partial u}{\partial x_1},\ldots, \frac{\partial u}{\partial x_{d-1}}\Big).
$$
With a slight abuse of notation, this is done both when $u=u(x')$ and $u=u(x)$. In the second case, we also write the usual gradient as $\nabla u=(\nabla'u, \tfrac{\partial u}{\partial x_d})$.

The symbol
\begin{equation}\label{defsigma}
    \sigma_{k}=\mathcal H^{k-1}(S^{k-1})
\end{equation} 
denotes the surface measure of the unit sphere in $\mathbb R^k$ (in our case, it will be either  $k=d$ or $k=d-1$).

We use the notation $a\sim b$, meaning that we consider indifferently the  quantities $a$ and $b$ in the argument considered; this convention will be used throughout the paper to illustrate approximation arguments.

\subsection{Dimension-reduction convergence}
In the following, $\omega$ denotes a bounded connected open subset of $\mathbb R^{d-1}$ with Lipschitz boundary, and for all $\e>0$
we define the {\em thin film}
$$
\Omega_\e=\omega\times (0,\e)\subset \mathbb R^d.
$$

In order to define a notion of convergence for functions $u_\varepsilon$ defined on $\Omega_\varepsilon$ as $\varepsilon \to 0$, we rescale them onto a fixed domain. To this end, we define functions $v_\varepsilon \in L^1(\Omega)$, with $\Omega = \omega \times (0,1)$, by
\begin{equation}\label{scaled-ve}
v_\varepsilon(x) = v_\varepsilon(x', x_d) := u_\varepsilon(x', \varepsilon x_d).
\end{equation}
Besides strong and weak convergence for $v_\varepsilon$ in the spaces $L^p(\omega \times (0,1))$, we will also use a notion of convergence for functions $u_\varepsilon \in L^1(\Omega_\varepsilon)$ toward a dimensionally reduced function $u \in L^1(\omega)$. It is customary to proceed as follows \cite{LDR,BFF}. 

\begin{definition}[Dimension-reduction convergence]\label{drco}
We say that $u_\e\in L^p(\Omega_\e)$ {\em (di\-men\-sion-reduction) converge in $L^p$} to $u\in L^p(\omega)$ as $\e\to0$, and we use the notation $u_\e\buildrel{\hbox {\rm\scriptsize DR }\, \, }\over{\longrightarrow} u$ in $L^p$, if:  

{\rm1)} the scaled functions $v_\e$ defined by \eqref{scaled-ve} converge in $L^p(\Omega)$ to some $v=v(x')$; that is, to some $v$  is independent of $x_d$; 

{\rm2)} the limit $u\in L^p(\omega)$ is defined by the equality $u(x')=v(x')$. 

\smallskip 
\noindent We will use a similar terminology if the scaled functions $v_\e$ converge to $v$ in a different topology; in particular, in $L^1_{\rm loc}$ or in $L^p$-weak. Moreover, we will often use equivalently $v$ and $u$ to denote the limit, even though they formally belong to different function spaces.
\end{definition}

In the case of thin films modeled by local energies in Sobolev spaces, this convergence is ensured by the following compactness result (see, for example, \cite{GCB}, Section 14.1). The proof relies on a simple application of Fubini’s theorem, which allows one to decouple the “vertical” and “planar” components of the gradient. 

\begin{lemma}[Dimension-reduction compactness for local functionals]\label{local-lemma}
Let $\{u_\e\}$ be a sequence with $u_\e\in H^1(\Omega_\e)$, and suppose that 
$$
\sup_\e\frac1\e\int_{\Omega_\e} |\nabla u_\e|^2dx<+\infty.
$$
Then, up to the addition of constants, $u_\e$ is precompact; that is, there exists $c_\e$ such that $u_\e+c_\e$ is precompact, with respect to the di\-men\-sion-reduction convergence in $L^2$, the limit  $u$ belongs to $H^1(\omega)$ and
$$
\int_\omega |\nabla'u|^2dx'\le \liminf_{\e\to 0}\frac1\e\int_{\Omega_\e} |\nabla u_\e|^2dx.
$$ 
\end{lemma}

\subsection{Fractional Sobolev spaces}
If $\Omega\subset \mathbb R^d$ is a bounded connected open set, the fractional Sobolev spaces $H^s(\Omega)$ are defined as the set of functions in $L^2(\Omega)$ such that their {\em Gagliardo seminorm}
$$
\lfloor u
\rfloor_s(\Omega)=\bigg(\int_{\Omega\times \Omega} \frac{|u(x)-u(y)|^2}{|x-y|^{d+2s}} dx\,dy\bigg)^{1/2}
$$
is finite (see \cite{leofrac,DPV}, to which we refer for notation and results). 

\smallskip
The space $H^1(\Omega)$ is a singular limit of the spaces $H^s(\Omega)$ in the following sense.
\begin{theorem}[Bourgain--Brezis--Mironescu limit theorem \cite{BBM,ponce}]\label{BBM-thm}
If $u_s$ is a family of functions with $u_s\in H^s(\Omega)$ and $\sup_s(1-s)\lfloor u\rfloor^2_{H^s(\Omega)}<+\infty$,
then, up to subsequences and the addition of constants, $u_s$ converges in $L^2(\Omega)$ as $s\to 1$ to a function $u\in H^1(\Omega)$. Furthermore, for $u\in H^1(\Omega)$ we have
$$
\Gamma\hbox{-}\lim_{s\to 1} (1-s)\int_{\Omega\times \Omega} \frac{|u(x)-u(y)|^2}{|x-y|^{d+2s}} dx\,dy=\frac{\sigma_d}{2d}\int_\Omega|\nabla u|^2dx,
$$
with $\sigma_d$ defined in \eqref{defsigma}. Furthermore, the limit holds also pointwise.
\end{theorem}

\section{\bf $\Gamma$-limit in the vertical direction} \label{asymp2}

We now start our analysis of the squared Gagliardo seminorms 
$$
\lfloor u
\rfloor^2_s(\Omega_\e)=\int_{\Omega_\e\times \Omega_\e} \frac{|u(x)-u(y)|^2}{|x-y|^{d+2s}} dx\,dy
$$
by considering the scaling for which the asymptotic behaviour is governed by the `thin direction' $x_d$. To give a heuristic idea of this scaling, 
we consider a function $u=u(x_d)$ and write the corresponding Gagliardo seminorm in terms of the function $v$ defined by $v(x_d)=u(\e x_d)$. After the change of variables $x=\e z$ and $y=\e w$, we have
\begin{equation}\label{2.2}
\lfloor u\rfloor_s^2(\Omega_\e)=\e^{d-2s}
\int_0^1\int_0^1(v(z_d)-v(w_d))^2\Big(\int_{\frac\omega\e\times\frac\omega\e}\frac{1}{|z-w|^{d+2s}}dz' dw'\Big)dz_d dw_d. 
\end{equation}
With $z_d-w_d\neq 0$ fixed, using the change of variable $\xi=z'-w'$ and taking into account the integrability of the kernel, we estimate
$$
\int_{\frac\omega\e\times\frac\omega\e}\frac{1}{|z-w|^{d+2s}}dz' dw'\sim \frac{|\omega|}{\e^{d-1}}\int_{\mathbb R^{d-1}}\frac{1}{(|z_d-w_d|^2+ |\xi|^2)^{\frac d2+s}}d\xi.
$$
We can then integrate out the dependence on $|z_d-w_d|$ by the change of variable $\xi=|z_d-w_d|\zeta$, which gives 
\[
\int_{\frac\omega\e\times\frac\omega\e}\frac{1}{|z-w|^{d+2s}}dz' dw'
\sim \left(\frac{|\omega|}{\e^{d-1}}\int_{\mathbb R^{d-1}}\frac1{(1+ |\zeta|^2)^{\frac d2+s}}d\zeta\right) \frac1{|z_d-w_d|^{1+2s}}.
\]
Plugging this relation back into \eqref{2.2} we obtain
\begin{equation}\label{2.3}
\lfloor u\rfloor_s^2(\Omega_\e)\ \sim\ \e^{1-2s}|\omega|\int_{\mathbb R^{d-1}}\frac1{(1+ |\zeta'|^2)^{\frac d2+s}}d\zeta'
\int_0^1\int_0^1\frac{(v(z_d)-v(w_d))^2}{|z_d-w_d|^{1+2s}}dz_d dw_d. 
\end{equation}
This formula suggests that

\smallskip
i) the first scaling is $\e^{1-2s}$;

\smallskip
ii) the limit can be expressed as the integral over $\omega$ of the one-dimensional Gagliardo 
$s$-seminorms along the vertical sections, multiplied by an appropriate constant.;

\smallskip
iii) The coefficient is given by a hypergeometric integral that encodes the interaction between the planar and the vertical variables.

\smallskip
To justify this argument, we first establish a preliminary result that allows us to extend the previous reasoning to functions depending also on the planar variable. More precisely, we prove matching lower and upper bounds for the 
$s$-seminorm in terms of the integral of the one-dimensional 
$s$-seminorm taken along the thin direction.

In the following, for $\tau>0$ small enough, we define the sets
\[\omega_\tau=\{x^\prime\in\omega: \ \hbox{\rm dist}(x^\prime,\partial\omega)>\tau\}\quad \mbox{and}\quad\  \Omega_\e^\tau=\omega_\tau\times (0,\e).\] 

\begin{lemma}[Estimate of the seminorm in the thin direction]\label{vert-lemma}
Let $s_\e\in(0,1)$, and let $\{u_\e\}_\e$ be a sequence in $H^{s_\e}(\Omega_\e)$ 
such that the following properties hold: 

\indent $\bullet$ (equiboundedness) $\sup_{\e>0}\|u_\e\|_\infty=L_\infty<+\infty$;  

\indent $\bullet$ (equi-planar Lipschitz condition) there exists $L>0$ such that for all $\e>0$ 
\begin{equation*}\label{equilip}
|u_\e(x^\prime, x_d)-u_\e(y^\prime, x_d)|\leq L|x^\prime-y^\prime| \ \ \hbox{\rm for all } x^\prime,y^\prime\in\omega, x_d\in (0,\e). 
\end{equation*}
Let $\tau>0$ and define
\begin{equation}\label{def-vertical}
V_\varepsilon(u_\varepsilon)
:=
\int_{\omega_\tau}
\int_0^\varepsilon
\int_0^\varepsilon
\frac{|u_\varepsilon(x',x_d)-u_\varepsilon(x',y_d)|^2}
{|x_d-y_d|^{1+2s_\varepsilon}}
\,dx_d\,dy_d\,dx'.
\end{equation}
For every $\eta>0$, the following estimates hold:
\begin{align}
(1+\eta)[u_\varepsilon]_{s_\varepsilon}^2(\Omega_\varepsilon)
+\frac{\varepsilon\phi(\varepsilon)}{\eta}
&\ge
C_{s_\varepsilon,d}\,V_\varepsilon(u_\varepsilon),
\label{vert-est1}
\\
[u_\varepsilon]_{s_\varepsilon}^2(\Omega_\varepsilon^\tau)
-\frac{\varepsilon\phi(\varepsilon)}{\eta}
&\le
(1+\eta)C_{s_\varepsilon,d}\,V_\varepsilon(u_\varepsilon).
\label{vert-est2}
\end{align}
The constant $C_{s,d}$ is defined by
\begin{equation}\label{defC}
C_{s,d}
=
\int_{\mathbb{R}^{d-1}}
\frac{1}{(1+|\xi|^2)^{\frac{d}{2}+s}}
\,d\xi.
\end{equation}
Moreover, the function $\phi$ depends only on $L_\infty$, $L$, $\omega$, and $\tau$, and satisfies
\[
\lim_{\varepsilon\to 0}\phi(\varepsilon)=0.
\]
\end{lemma}

\begin{proof} The idea of the proof is to decompose the interaction energy into a vertical
component, corresponding to interactions along the thin direction, and a horizontal
component, which will be shown to be negligible as $\e\to 0$.

We extend $u_\e$ to $\mathbb R^{d-1}\times(0,\e)$ by setting $u_\e(x^\prime,x_d)=0$ for $x^\prime\in\mathbb R^{d-1}\setminus\omega$. 
For each measurable subset $A$ of $\mathbb R^{d-1}$, we define 
\begin{eqnarray*}
&&I_{\e,\tau}(A)= \int_{\omega_\tau}\int_0^\e\int_0^\e\int_A\frac{|u_\e(x^\prime, x_d)-u_\e(y^\prime, y_d)|^2}{(|x_d-y_d|^2+|x^\prime-y^\prime|^2)^{\frac{d}{2}+s_\e}}\, dy^\prime\, dx_d\, dy_d\, dx^\prime\\
&&I^{\rm v}_{\e,\tau}(A)=\int_{\omega_\tau}\int_0^\e\int_0^\e\int_A \frac{|u_\e(x^\prime, x_d)-u_\e(x^\prime, y_d)|^2}{(|x_d-y_d|^2+|x^\prime-y^\prime|^2)^{\frac{d}{2}+s_\e}}\, dy^\prime\, dx_d\, dy_d\, dx^\prime\\ 
&&I^{\rm h}_{\e,\tau}(A)=\int_{\omega_\tau}\int_0^\e\int_0^\e\int_A \frac{|u_\e(y^\prime, y_d)-u_\e(x^\prime, y_d)|^2}{(|x_d-y_d|^2+|x^\prime-y^\prime|^2)^{\frac{d}{2}+s_\e}}\, dy^\prime\, dx_d\, dy_d\, dx^\prime  . 
\end{eqnarray*}
Note that, even though the extension may not belong to $H^{s_\varepsilon}(\mathbb R^{d-1}\times (0,\varepsilon))$, the above quantities are finite. This follows observing that $|x^\prime - y^\prime| > \tau$ whenever $y^\prime \notin \omega$ and $x^\prime \in \omega_\tau$ and using the equi-boundedness of   $u_{\e}$. 

Since the change of variable $z^\prime-y^\prime=|z_d|\xi$ gives the equality 
\begin{equation}
\int_{\mathbb R^{d-1}} \frac{1}{(|z_d|^2+|z^\prime-y^\prime|^2)^{\frac{d}{2}+s}}\, dy^\prime=
\int_{\mathbb R^{d-1}} \frac{|z_d|^{d-1}}{(|z_d|^2+(|z_d||\xi|)^2)^{\frac{d}{2}+s}}\, d\xi
=
\frac{C_{s,d}}{|z_d|^{1+2s}}
\label{rd1}
\end{equation} 
for all $z_d\neq 0$ and $z^\prime\in\mathbb R^{d-1}$,  
then using Fubini's Theorem we can write  
\begin{equation}
I^{\rm v}_{\e,\tau}(\mathbb R^{d-1})=C_{s_\e,d}\int_{\omega_\tau}\int_0^\e\int_0^\e\frac{|u_\e(x^\prime, x_d)-u_\e(x^\prime, y_d)|^2}{|x_d-y_d|^{1+2s_\e}}\, dx_d\, dy_d\, dx^\prime.\label{cvar}
\end{equation}
A triangular inequality gives 
\begin{eqnarray}
&&I_{\e,\tau}^{\rm v}(\mathbb R^{d-1})\leq (1+\eta)I_{\e,\tau}(\mathbb R^{d-1})+\frac{1}{\eta}I_{\e,\tau}^{\rm h}(\mathbb R^{d-1})\label{tr1}\\
&&I_{\e,\tau}(\mathbb R^{d-1})\leq (1+\eta)I_{\e,\tau}^{\rm v}(\mathbb R^{d-1})+\frac{1}{\eta}I_{\e,\tau}^{\rm h}(\mathbb R^{d-1})\label{tr2}
\end{eqnarray}
for all $\eta>0$. 

We now show that 
\begin{equation}\label{stima-or}
I_{\e,\tau}^{\rm h}(\mathbb R^{d-1})\leq\e\phi(\e),
\end{equation}
where $\phi$ depends on $L$, $\omega$ and $\tau$, and $\lim\limits_{\e\to 0}\phi(\e)=0$.
Noting that for all $x^\prime\in\omega_\tau$ 
\begin{equation*}
\int_{\mathbb R^{d-1}\setminus\omega} \frac{1}{(|x_d-y_d|^2+|x^\prime-y^\prime|^2)^{\frac{d}{2}+s_\e}} dy^\prime
\leq\int_{\mathbb R^{d-1}\setminus B_\tau(0)} \frac{1}{|\xi^\prime|^{d+2s_\e}}\, d\xi^\prime
=\frac{\sigma_{d-1}}{(1+2s_\e)\tau^{1+2s_\e}}, 
\end{equation*}
where $\sigma_{d-1}=\mathcal H^{d-2}(S^{d-2})$, 
we get 
\begin{equation}
I^{\rm h}_{\e,\tau}(\mathbb R^{d-1}\setminus \omega)
\leq \frac{\sigma_{d-1}L^2_\infty |\omega_\tau|}{\tau^{1+2s_\e}}\e^2\label{st1}\\
\end{equation}
Hence, we only have to estimate the integral in the set where $y^\prime\in\omega$. 
We have 
\begin{eqnarray*}
I^{\rm h}_{\e,\tau}(\omega)
&\leq& 
2L^2\int_{\omega_\tau}\int_0^\e\int_{\omega}|x^\prime-y^\prime|^2\int_0^\e\frac{1}{(t^2+|x^\prime-y^\prime|^2)^{\frac{d}{2}+s}}\, dt\, dy^\prime\, dy_d\, dx^\prime\\
&\leq& 
2L^2\e\phi(\e)
\end{eqnarray*}
where $L$ is the Lipschitz constant and 
\begin{equation}\label{defphi}\phi(\e)=
\int_{\omega}\int_{\omega}|x^\prime-y^\prime|^2
\int_0^\e\frac{1}{(t^2+|x^\prime-y^\prime|^2)^{\frac{d}{2}+s}}\, dt
\, dy^\prime\, dx^\prime.
\end{equation}
Since the integrand in \eqref{defphi} pointwise converges to $0$ almost everywhere, and it is estimated by 
\begin{eqnarray*}
|x^\prime-y^\prime|^2\int_0^\e\frac{1}{(t^2+|x^\prime-y^\prime|^2)^{\frac{d}{2}+s}}\, dt
&\leq&|x^\prime-y^\prime|^{3-d-2s_\e} \int_0^{+\infty}\frac{1}{(1+w^2)^{\frac{d}{2}+s_\e}} d w \\
&\leq&\hbox{\rm diam}(\omega)^{1-2s_\e} |x^\prime-y^\prime|^{2-d}\int_0^{+\infty}\frac{1}{(1+w^2)^{\frac{d}{2}+s_\e}} d w, 
\end{eqnarray*}
with
$$\int_{\omega}\int_{\omega}|x^\prime-y^\prime|^{2-d}\, dx^\prime\, dy^\prime<+\infty,$$ 
by Lebesgue's Theorem we obtain 
$\lim\limits_{\e\to 0}\phi(\e)=0$.
Since \eqref{st1} holds, this implies \eqref{stima-or}.

Noting that \eqref{st1} holds also for $I_{\e,\tau}(\mathbb R^{d-1}\setminus\omega)$, we can write
\begin{equation}\label{stima-semi}
I_{\e,\tau}(\mathbb R^{d-1})
= I_{\e,\tau}(\omega)+r_\e, 
\end{equation}
with the remainder $r_\e$ satisfying $|r_\e|\leq \sigma_{d-1}L^2_\infty |\omega_\tau|\tau^{-1-2s_\e}\e^2$. 

Recalling \eqref{tr1} and \eqref{tr2}, estimates \eqref{stima-or} and \eqref{stima-semi} imply that 
\begin{eqnarray*}
I_{\e,\tau}^{\rm v}(\mathbb R^{d-1})&\leq&(1+\eta)(I_{\e,\tau}(\omega)+|r_\e|)+\frac{\e\phi(\e)}{\eta}\\
&\leq&(1+\eta)\lfloor u_\e\rfloor_{s_\e}^2(\Omega_\e)
+C\e^2\tau^{-1-2s_\e}
+\frac{\e\phi(\e)}{\eta},
\\
\lfloor u_\e\rfloor_{s_\e}^2(\Omega_\e^\tau)-C\e^2\tau^{-1-2s_\e}&\leq& I_{\e,\tau}(\omega)-|r_\e|\\
&\leq&(1+\eta)I_{\e,\tau}^{\rm v}(\mathbb R^{d-1})+\frac{\e\phi(\e)}{\eta}, 
\end{eqnarray*}
and, since \eqref{cvar} holds, the claim follows. \end{proof}

It is useful to provide a more explicit expression for the constants $C_{s,d}$ in \eqref{defC}, as stated in the following proposition. We note that the computation of these constants, as well as of others appearing below, relies on the evaluation of the integral
\[
\int_0^{+\infty}\frac{1}{(1+t^2)^s}\,dt
=
\sqrt{\pi}\,\frac{\Gamma\!\left(s-\frac12\right)}{2\Gamma(s)},
\]
which holds for $s>\frac12$. Here $\Gamma$ denotes the usual Euler's $\Gamma$ function.

\begin{proposition}[Alternative expression for $C_{s,d}$]\label{pro-Csd}
We have 
\begin{align}\label{csd-1}
C_{s,d}=\int_{\mathbb{R}^{d-1}}
\frac{1}{(1+|\xi|^2)^{\frac{d}{2}+s}}
\,d\xi=\frac{2}{\sqrt\pi} \,
\frac{\Gamma(\frac{d}{2}+1)\Gamma(\frac{1}{2}+s)}{\Gamma(\frac{d}{2}+s)}\, \frac{\sigma_d}{2d}=\pi^{\frac{d-1}{2}} \frac{\Gamma(\frac{1}{2}+s)}{\Gamma(\frac{d}{2}+s)}. 
\end{align}
\end{proposition}
\noindent For the sake of completeness, we include the proof in the Appendix.

\medskip
Observing that
\[
\int_\omega\int_0^\e\!\!\int_0^\e
\frac{|u(x^\prime,x_d)-u(x^\prime,y_d)|^2}{|x_d-y_d|^{1+2s}}
\, dx_d\, dy_d\, dx^\prime
=
\e^{1-2s}
\int_\omega\int_0^1\!\!\int_0^1
\frac{|v(x^\prime,t)-v(x^\prime,\tau)|^2}{|t-\tau|^{1+2s}}
\, dt\, d\tau\, dx^\prime,
\]
where $v$ denotes the rescaled function defined by 
\[
v(x^\prime,t)=u(x^\prime,\e t),
\]
the previous lemma shows that the natural scaling factor is $\e^{1-2s_\e}$. 

\smallskip
Using the estimates in Lemma \ref{vert-lemma}, we can now prove the $\Gamma$-convergence result in this scaling. The $\Gamma$-limit is computed with respect to the weak convergence in $L^2_{\rm loc}(\omega\times(0,1))$ of the rescaled functions $v_\e$, under the assumption that they are bounded in $L^2_{\rm loc}(\omega\times(0,1))$. As in the integer case; that is, for the Dirichlet integral on $H^1(\Omega)$, this property is not guaranteed, up to addition of constants, since we only have a control in the $x_d$-direction. However, such a condition follows from a control of $u$ at $x_d=0$; e.g., prescribing boundary conditions $u(x',0)=u_0(x')$.

\begin{theorem}[Gamma-limit at the first scaling]\label{fulldim} 
Let $s_0\in[0,1]$ be fixed, and $\{s_\e\}_\e\subset (0,1)$ be such that $s_\e\to s_0$ as $\e\to 0$. 
Let the functional $E_\e$ be defined in $H^{s_\e}(\Omega_\e)$ by 
$$E_\e(u)=\frac{1}{\e^{1-2s_\e}}\lfloor u\rfloor_{s_\e}^2(\Omega_\e).$$ 
Then the following 
$\Gamma$-convergence results hold with respect to the weak convergence in $L^2_{
\rm loc}(\omega\times (0,1))$ of the corresponding scaled functions:

\indent $\bullet$ if $s_0<1$, then 
$$\Gamma\hbox{\rm -}\lim_{\e\to 0}{E_\e}({u})=C_{s_0,d}\int_{\omega}\int_0^1\int_0^1\frac{|v(x^\prime, x_d)-v(x^\prime, y_d)|^2}{|x_d-y_d|^{1+2s_0}}\, dx_d\, dy_d\, dx^\prime,$$
\hspace{8mm} with $C_{s_0,d}$ is defined in  \eqref{defC}; 

\smallskip 

\indent $\bullet$ {\rm(Bourgain--Brezis--Mironescu-type result)} if $s_0=1$, then 
\begin{equation*}\Gamma\hbox{\rm -}\lim_{\e\to 0}(1-s_\e){E_\e}(u)
= \frac{\sigma_d}{2d} \int_{\omega}\int_0^1 \Big|\frac{\partial v}{\partial x_d}(x^\prime,x_d)\Big|^2\,dx_d\, dx^\prime, 
\end{equation*}
\smallskip 
where we have used \eqref{csd-1} to write $C_{1,d}=\frac{\sigma_d}{2d}$.
\end{theorem}

\begin{proof}
Let $u_\e$ be such that the corresponding scaled functions $v_\e$ weakly converge to $v$ in $L^2_{
\rm loc}(\omega\times (0,1))$. 
Note that we can suppose that the sequence $\{u_\e\}$ is equibounded in $L^\infty(\Omega_\e)$ up to a truncation argument, since the functionals $E_\e$ decrease by truncation. 
Let $\tau>0$ and $\varphi_\tau\colon\mathbb R^{d-1}\to [0,+\infty)$ be a mollifier with support in $B_\tau(0)$. 
We set $u_\e^\tau(x^\prime,x_d)=u_\e \ast\varphi_\tau$, where the convolution is performed in the variable $x^\prime$. Then, the sequence $\{u_\e^\tau\}_\e$ satisfies the uniform Lipschitz condition required to apply Lemma \ref{vert-lemma}. Noting that 
$$\lfloor u_\e\rfloor_{s_\e}^2(\Omega_\e)\geq \lfloor u_\e^\tau\rfloor_{s_\e}^2(\omega_\tau\times(0,\e)),$$ 
we can apply Lemma \ref{vert-lemma} to the sequence $\{u_\e^\tau\}_\e$ with $\Omega_\e$ replaced by $\Omega_\e^\tau=\omega_{\tau}\times (0,\e)$. Then, by \eqref{vert-est1} and the arbitrariness of $\eta>0$ therein, noting that $\e^{2s_\e}\phi(\e)\to 0$ we get 
\begin{eqnarray*}
\liminf_{\e\to 0} E_\e(u_\e)&\geq&C_{s_0,d} \liminf_{\e\to 0}\e^{2s_\e-1}\int_{\omega_{2\tau}}\int_0^\e\int_0^\e\frac{|u^\tau_\e(x^\prime, x_d)-u^\tau_\e(x^\prime, y_d)|^2}{|x_d-y_d|^{1+2s_\e}}\, dx_d\, dy_d\, dx^\prime\\
&=&C_{s_0,d} \liminf_{\e\to 0}\int_{\omega_{2\tau}}\int_0^1\int_0^1\frac{|v^\tau_\e(x^\prime, x_d)-v^\tau_\e(x^\prime, y_d)|^2}{|x_d-y_d|^{1+2s_\e}}\, dx_d\, dy_d\, dx^\prime, 
\end{eqnarray*}
where $v^\tau_\e(x^\prime, x_d)=u^\tau_\e(x^\prime, \e x_d)$. 
If we also set $v^\tau=v\ast \varphi_\tau$, we have the convergence $v^\tau_\e\to v^\tau$ in $L^2(\omega_\tau\times (0,1))$. 
We now separately consider the cases $s_0<1$ and $s_0=1$. 

\smallskip 

$\bullet$ {Case $s_0\in[0,1)$.} By Fatou's Lemma, we get 
\begin{eqnarray*}
&&\hspace{-2cm}\liminf_{\e\to 0}\int_{\omega_{2\tau}}\int_0^1\int_0^1\frac{|v^\tau_\e(x^\prime, x_d)-v^\tau_\e(x^\prime, y_d)|^2}{|x_d-y_d|^{1+2s_\e}}\, dx_d\, dy_d\, dx^\prime\\
&&\geq 
\int_{\omega_{2\tau}}\int_0^1\int_0^1\frac{|v^\tau(x^\prime, x_d)-v^\tau(x^\prime, y_d)|^2}{|x_d-y_d|^{1+2s_0}}\, dx_d\, dy_d\, dx^\prime. 
\end{eqnarray*}
Taking the limit as $\tau\to 0$ we obtain the lower bound. 

The upper bound for a  function $v\colon\mathbb R^{d-1}\times(0,1)\to\mathbb R$ such that  
\begin{equation*}
|v(x^\prime, x_d)-v(y^\prime, x_d)|\leq L|x^\prime-y^\prime|
\end{equation*}
for all $x^\prime,y^\prime\in \mathbb R^{d-1}$ and $x_d\in (0,1)$, is achieved by the trivial recovery sequence $u_\e(x^\prime, x_d)=v(x^\prime, \frac{x_d}{\e})$. Indeed, estimate \eqref{vert-est2} gives 
$$E_\e(u_\e)\leq 
\frac{\e^{2s_\e}\phi(\e)}{\eta}+ (1+\eta)C_{s_\e,d}\int_{\omega_\tau}\int_0^1\int_0^1\frac{|v(x^\prime, x_d)-v(x^\prime, y_d)|^2}{|x_d-y_d|^{1+2s_\e}}\, dx_d\, dy_d\, dx^\prime,$$
allowing us to obtain the upper bound. 
For a general $v\in L^2(\omega\times (0,1))$ we can proceed by approximation.

\smallskip  

$\bullet$ {Case $s_0=1$.} Again by Fatou's Lemma, 
\begin{eqnarray*}
&&\hspace{-2cm}\liminf_{\e\to 0}\int_{\omega_{2\tau}}(1-s_\e)\int_0^1\int_0^1\frac{|v^\tau_\e(x^\prime, x_d)-v^\tau_\e(x^\prime, y_d)|^2}{|x_d-y_d|^{1+2s_\e}}\, dx_d\, dy_d\, dx^\prime\\
&&\geq 
\int_{\omega_{2\tau}}\int_0^1 \Big|\frac{\partial v^\tau}{\partial x_d}(x^\prime,x_d)\Big|^2\, dx_d \, dx^\prime, 
\end{eqnarray*}
where we used the $\Gamma$-convergence result for the $s$-seminorms of Bourgain, Brezis and Mironescu, Theorem \ref{BBM-thm}. Taking the limit as $\tau\to 0$ we obtain the lower bound. 

\smallskip
For the upper bound, we can proceed exactly as in the case $s_0<1$, by considering the trivial recovery sequence $u_\e(x^\prime, x_d)=v(x^\prime, \frac{x_d}{\e})$ for a function $v$ which is Lipschitz with respect to the planar variable. Using \eqref{vert-est2} multiplied by $1-s_\e$, 
we get 
$$(1-s_\e)E_\e(u_\e)\leq 
\frac{\e^{2s_\e}\phi(\e)}{\eta}(1-s_\e)+ (1+\eta)C_{s_\e,d} \int_{\omega_\tau}(1-s_\e)
\lfloor v(x^\prime,\cdot) \rfloor^2_{s_\e}(0,1)dx^\prime.$$
Using the pointwise convergence in Theorem \ref{BBM-thm} we obtain the upper estimate. \end{proof}

\section{The dimension-reduction regimes}
We now describe the behaviour of Gagliardo seminorms in the dimension-reduction regimes. We first determine two different scaling laws for $s<1/2$ and $s>1/2$, and then prove dimension-reduction compactness, in the first case with respect to a weak convergence, and in the second case with respect to a strong convergence, and compute the $\Gamma$-limits with respect to those convergences. The case $s=1/2$ is critical and is obtained as a limit case of both other regimes.

\subsection{Determination of the scalings}\label{heu}
The heuristic argument used to identify the next scaling after 
$\varepsilon^{1-2s_\varepsilon}$ in the asymptotic analysis—namely, 
the scaling at which a dimension-reduction effect appears—is the following. 
Interactions between points at distance of order $\varepsilon$ or smaller 
can be regarded as genuinely $d$-dimensional, whereas interactions between 
points at distance larger than $\varepsilon$ are essentially $(d-1)$-dimensional.

These two types of interaction scale differently depending on whether 
$s > \tfrac12$ (the high-integrability regime) or $s < \tfrac12$ 
(the low-integrability regime). In the former case, the dominant 
contributions come from interactions at scales not larger than 
$\varepsilon$, while in the latter case the dominant interactions 
occur at scales greater than or equal to $\varepsilon$.

\smallskip
We now illustrate the heuristic argument, giving an estimate of the pointwise limit of  $\lfloor u\rfloor_{s_\e}^2(\Omega_\e)$ when
$u(x',x_d)=v(x')$ with $v$ a Lipschitz function and $s_\e\to s_0$, where $s_0\in [0,1]$ is fixed. We can limit our analysis to $v(x')=x_1$. In order to give an upper bound, for fixed $x\in \Omega_\e$ we can subdivide interactions inside the ball $B_{K\e}(x)$, where $K>1$ is any constant, and those outside the cylinder $\{y\in \Omega_{\e}\ |\ |x'-y'|<(K-1)\e\}$, since
$$
\Omega_\e\subset B_{K\e}(x)\times C_{(K-1)\e}(x),\qquad  C_{(K-1)\e}(x)=\{y\in\Omega_\e: |x'-y'|>(K-1)\e\}.
$$
We then have 
\begin{eqnarray*}
\int_{B_{K\e}(x)\cap \Omega_\e}\!\frac{(u(x)-u(y))^2}{|x-y|^{d+2s_\e}}dy
&\le& \int_{B_{K\e}(x)}\!\frac{|x_1-y_1|^2}{|x-y|^{d+2s_\e}}dy
\,\, \le \,\, \int_{B_{K\e}(x)}\frac{1}{|x-y|^{d+2s_\e-2}}dy\\
&=&\int_{B_{K\e}}\frac{1}{|\xi|^{d+2s_\e-2}}d\xi
\,\, =\,\, C\int_0^{K\e}t^{1-2s_\e}dt\\
&\le& CK^{2-2s_\e}\frac{\e^{2-2s_\e}}{1-s_\e},
\end{eqnarray*}
so that
\begin{equation}\label{geu-1}
\int_{\Omega_\e}\int_{B_{K\e}(x)\cap \Omega_\e}\frac{(u(x)-u(y))^2}{|x-y|^{d+2s_\e}}dy\le CK^{2-2s_\e}|\omega|\frac{\e^{3-2s_\e}}{1-s_\e},
\end{equation}
with $C$ depending only on $d$. Note that the term $K^{2-2s_\e}$ is largely suboptimal for $K$ large.

As far as the interactions outside the cylinder are concerned, if $R$ denotes the diameter of $\omega$, we have
\begin{eqnarray}\label{geu-1.5}\nonumber
\int_{ C_{(K-1)\e}(x)}\frac{(u(x)-u(y))^2}{|x-y|^{d+2s_\e}}dy
&\le& C\int_{C_{(K-1)\e}(x)}\frac{1}{|x'-y'|^{d+2s_\e-2}}dy\\
&=&\nonumber
C\e\int_{\{(K-1)\e<|\xi|<R\}}\frac{1}{|\xi'|^{d+2s_\e-2}}d\xi'\\
&=&\nonumber C\e\int_{(K-1)\e}^{R}t^{-2s_\e}dt
\\
&\le&C
\e\frac{((K-1)\e)^{1-2s_\e}-R^{1-2s_\e}}{2s_\e-1}.
\end{eqnarray}

\medskip
\noindent {\bf The high-integrability regime.}
We now consider the case $s_0>1/2$,
so that, from \eqref{geu-1.5}, we have 
\begin{equation}
\int_{\Omega_\e}\int_{ C_{(K-1)\e}(x)}\frac{(u(x)-u(y))^2}{|x-y|^{d+2s_\e}}dy
\le C(K-1)^{1-2s_\e}|\omega|\frac{\e^{3-2s_\e}}{2s_\e-1},
\end{equation}
with $C$ depending only on $d$, and, together with \eqref{geu-1}, that 
\begin{equation}\label{geu-2}
\lfloor u\rfloor_{s_\e}^2(\Omega_\e)\le C|\omega|\e^{3-2s_\e}\Big(
\frac{K^{2-2s_\e}}{1-s_\e}+\frac{(K-1)^{1-2s_\e}}{2s_\e-1}\Big).
\end{equation}

On the other hand, we can give a lower bound by considering  only $y\in B_{\e/4}(x)$  for $x\in 
\Omega_\e$ with dist$(x',\partial\omega)>\e/2$ and $ x_d\in[\e/4,3\e/4]$, so that 
\begin{eqnarray*}
&&\hskip-1cm\int_{B_{\e/4}(x)\cap \Omega_\e}\frac{(u(x)-u(y))^2}{|x-y|^{d+2s_\e}}dy
= \int_{B_{\e/4}(x)}\frac{|x_1-y_1|^2}{|x-y|^{d+2s_\e}}dy\\
&=& \frac1d\int_{B_{\e/4}(x)}\frac{1}{|x-y|^{d+2s_\e-2}}dy
=\frac1d\int_{B_{\e/4}}\frac{1}{|\xi|^{d+2s_\e-2}}d\xi
\\
&=&\frac{C}d\int_0^{\e/2}t^{1-2s_\e}dt= C\frac{\e^{2-2s_\e}}{1-s_\e},
\end{eqnarray*}
which shows that
\begin{equation}\label{geu-3}
\lfloor u\rfloor_{s_\e}^2(\Omega_\e)\ge C|\omega|\frac{\e^{3-2s_\e}}{1-s_\e},
\end{equation}
with $C$ depending only on $d$. 

\smallskip
Estimates \eqref{geu-2} and \eqref{geu-3} suggest that the relevant scaling of the energy be 

\smallskip
1) $\e^{3-2s_\e}$ for $s_0\in (1/2,1)$

\smallskip
2) $\frac{\e^{3-2s_\e}}{1-s_\e}$
for $s_0=1$.

\medskip
\noindent {\bf The low-integrability regime.}
We now turn to the case $s_0<1/2$,
in which the sign of the denominator in \eqref{geu-1.5} changes, so that we have 
\begin{equation}
\int_{\Omega_\e}\int_{ C_{(K-1)\e}(x)}\frac{(u(x)-u(y))^2}{|x-y|^{d+2s_\e}}dy
\le C
\e^2|\omega|\frac{R^{1-2s_\e}}{1-2s_\e}
\end{equation}
with $C$ depending only on $d$. This estimate, together with \eqref{geu-1}, implies that 
$$
\lfloor u\rfloor_{s_\e}^2(\Omega_\e)\le C|\omega|
\Big(
\e^2\frac{R^{1-2s_\e}}{1-2s_\e}+\e^{3-2s_\e}
\frac{K^{2-2s_\e}}{1-s_\e}\Big)
\le C|\omega|R^{1-2s_0}\e^2,
$$
where this last constant $C$ also depends on $s_0$.

To check that the scaling is sharp, with fixed $\delta>0$, for all $x'\in\omega$ with dist$(x',\partial\omega)>\delta$ we estimate
$$
\int_{\{y\in\Omega_\e:K\e<|x-y|<\delta\}}\frac{(u(x)-u(y))^2}{|x-y|^{d+2s_\e}}dy\ge C\e \int_{K\e}^\delta 
t^{-2s_\e}dt=C\e\frac{\delta^{1-2s_\e}-(K\e)^{1-2s_\e}}{1-2s_\e},
$$
so that $\lfloor u\rfloor_{s_\e}^2(\Omega_\e)
\ge C\e^2$,
with this last constant also depending on $\delta$. Hence, in this regime the relevant scaling is $\e^2$.

\medskip
\noindent {\bf The critical regime.} At $s=1/2$ the interactions at scale up to $\e$ and larger than any fixed $\delta>0$ are negligible. As a consequence, the correct scaling turns out to be $\e^2|\log\e|$, and the limit is local. Note that, compared with the other regimes, when $s=1/2$ we have $\e^2=\e^{3-2s}$; however, at this scaling the functionals diverge. The $\Gamma$-limit is obtained by optimizing the arguments from the cases $s\neq 1/2$.

\subsection{$\Gamma$-limit in the high-integrability case}\label{dr+12}
In the case where $s_\varepsilon \to s_0 > \tfrac12$, we prove a dimension-reduction theorem with respect to strong convergence in $L^1$ for sequences of functions such that
\[
\frac{1 - s_\varepsilon}{\varepsilon^{3 - 2s_\varepsilon}} \, \lfloor u_\varepsilon \rfloor_{s_\varepsilon}^2(\Omega_\varepsilon)
\]
is equibounded, even though the factor $1 - s_\varepsilon$ is actually redundant unless $s_\varepsilon \to 1$. 
We show that the cluster points of such sequences belong to $H^1(\omega)$, and that the corresponding $\Gamma$-limit is a multiple of the Dirichlet integral.

\subsubsection{Strong dimension-reduction compactness}

 Lemma \ref{local-lemma}
 is based on the fact that, in the notation of Definition \ref{drco},
we may write
\begin{eqnarray*}
\frac1\e\int_{\Omega_\e} |\nabla u_\e|^2dx=\int_{\omega\times(0,1)}|\nabla' v_\e|^2dx+\frac1{\e^2} \int_{\omega\times(0,1)}\Big|\frac{\partial v_\e}{\partial x_d}\Big|^2dx
\ge \int_{\omega\times(0,1)}|\nabla v_\e|^2dx,
\end{eqnarray*}
which at the same time proves compactness for $v_\e$ in $H^1(\omega\times (0,1))$ and that $\frac{\partial v_\e}{\partial x_d}$ tends to~$0$.
The decoupling of the ``horizontal'' and ``vertical'' derivatives is not straightforward for Gagliardo seminorms. 
To overcome this difficulty, in the following result we adopt a discretization procedure for the Gagliardo seminorm, which allows us to construct sequences in local Sobolev spaces. 
The scaling argument can then be applied to these sequences.

\begin{theorem}[Non-local dimension-reduction compactness]\label{nlcth-1}
Let $s=s_\e\to s_0>1/2$ and let $u_\e$ be such that $$\sup_\e\frac{1-s}{\e^{3-2s}}\lfloor u_\e\rfloor_s^2(\Omega_\e)\le S<+\infty.$$Then there exist $u\in H^1(\omega)$ and a subsequence $\{u_{\e_j}\}_j$ such that, up to the addition of constants, 
$u_{\e_j}\buildrel{\hbox {\rm\scriptsize DR }\, \, }\over{\longrightarrow} u$ locally in $L^1$.
\end{theorem}

\begin{proof} To make the proof easier to follow, we divide it into several steps.

\smallskip
{\bf Step 1.} {\em A first change of variables with a bound on the range of interactions.}

\smallskip
We fix $r\in (0,1/2)$ and, in view of the heuristic computation in Section \ref{heu}, we limit our analysis to interactions with $|x-y|<\e r$; that is, using the change of variables $\xi=y-x$, we estimate
\begin{equation}\label{Proof1}
\lfloor u_\e\rfloor_s^2(\Omega_\e)
\ge \int_{\{|\xi|<\e r\}}\frac1{|\xi|^{d+2s-2}}\int_{\Omega_\e^r}\frac{|u(x+\xi)-u(x)|^2}{|\xi|^2}dx d\xi,
\end{equation}
where $\Omega_\e^r=\{x\in\Omega_\e:{\rm dist}(x,\partial \Omega_\e)> r\e\}$.

\medskip
{\bf Step 2.} {\em Parameterization on the {\em Stiefel} manifold.}

\smallskip
The inner integral in \eqref{Proof1} can be interpreted as a difference quotient, which will yield suitable piecewise-affine interpolations.

Before addressing the general case $d \geq 2$, we first illustrate the strategy in the simpler case $d=2$. In this situation, each vector $\xi$ can be completed to the orthogonal basis $\{\xi, \xi^\perp\}$. By symmetry, we obtain that 
$$\lfloor u_\e\rfloor_s^2(\Omega_\e)\geq \int_{\{|\xi|<\e r\}}\frac1{|\xi|^{d+2s-2}}F_\e^\xi(u_\e)\, d\xi,$$ 
where 
$$F_\e^\xi(u_\e)=\frac{1}{2}\int_{\Omega_\e^r}\frac{|u(x+\xi)-u(x)|^2+|u(x+\xi^\perp)-u(x)|^2}{|\xi|^2}dx. $$
The functionals $F_\e^\xi$ can be discretized on suitable square lattices oriented with $\xi$. 

In dimension larger than two, the extension of $\xi$ to an orthogonal basis is performed using the  
 {\em set of orthonormal bases} (Stiefel manifold) of $\mathbb{R}^d$
\[
V:=\{\overline{\nu}=(\nu_1,...,\nu_d) : \nu_j \in S^{d-1} \text{ such that } \langle\nu_i, \nu_j\rangle=0 \text{ for } i\neq j \}.
\]
We observe that $V$ has Hausdorff dimension equal to $k_d:=d(d-1)/2$. 

Following the notation in \cite[page 3445]{Solci-vortices} (see also \cite[Remark 5]{BBD}), we write $\xi=\varepsilon r\rho\,\nu_n$ for some basis element $\nu_n$ of $\overline{\nu}\in V$, and rewrite \eqref{Proof1} as
\begin{eqnarray}\label{Proof2}
\lfloor u_\varepsilon\rfloor_s^2(\Omega_\varepsilon)
\geq
\frac{(\varepsilon r)^{2-2s}}{\mathcal{H}^{k_d}(V)}
\int_0^1 \frac{\sigma_d}{\rho^{-1+2s}}
\int_V \frac{1}{d}\sum_{n=1}^d
\int_{\Omega_\varepsilon^r}
\frac{\bigl|u_\varepsilon(x+\varepsilon r\rho\,\nu_n)-u_\varepsilon(x)\bigr|^2}{|\varepsilon r\rho|^2}
\,dx\, d\mathcal{H}^{k_d}(\overline{\nu})\, d\rho .\nonumber\\
\end{eqnarray}
This follows from a standard change of variables; the same argument (with full details) is given in \cite[Section 3.1]{BBD}.

\medskip
{\bf Step 3.} {\em Discretization on lattices with lattice size up to $\e r$.}

\smallskip
Given $\rho>0$ and $\overline{\nu}\in V$, we define $$\mathbb Z^d_{\rho\overline{\nu}}:=\{z_1\rho\nu_1+z_2\rho\nu_2+...+z_d\rho\nu_d : (z_1,...,z_d)\in \mathbb Z^d\}$$ and $Q_{\rho\overline{\nu}}$ as the cube generated by the orthogonal basis $\{\rho\nu_1,...,\rho\nu_d\}$. With fixed $\rho\in(0,1)$, we set 
\[
\mathcal{I}^{\e r}_{\rho\overline{\nu}}(\Omega):=\{k\in \e r\mathbb Z_{\rho\overline{\nu}}^d : k+\e r Q_{\rho\overline{\nu}} \subset \subset \Omega\},
\]
and for every $k\in \mathcal{I}^{\e r}_{\rho\overline{\nu}}(\Omega)$ we define
    \[
    u_\e^{\rho\overline{\nu}}(k)= \displaystyle \frac{1}{|\e r\rho|^d}\int_{k+\e rQ_{\rho\overline{\nu}}} u_\e\,dx.
\] 
By Jensen's inequality, we then have
\begin{eqnarray*}\nonumber
&&\hskip-2cm\int_{\Omega_\e^r}\sum_{n=1}^d\frac{|u_\e(x+\e r\rho\nu_n)-u_\e(x)|^2}{|\e r\rho|^2} dx\\
 &\ge&\sum_{k\in \e r \mathbb Z^d_{\rho\overline{\nu}}}\sum_{n=1}^d|\e r\rho|^d\frac{|u_\e^{\rho\overline{\nu}}(k+\nu_n)-u_\e^{\rho\overline{\nu}}(k)|^2}{|\e r\rho|^2} dx.
\end{eqnarray*}

{\bf Step 4.} {\em Estimate with piecewise-affine interpolations.}

\smallskip
We still let $u_\e^{\rho\overline{\nu}}$ denote the piecewise-affine interpolation of the discrete function $u_\e^{\rho\overline{\nu}}$ defined on $\mathcal{I}^{\e r}_{\rho\overline{\nu}}(\Omega)$, extended through a standard Kuhn decomposition (see, e.g.~\cite[Remark~6]{BBD}). This function is well defined on a slightly smaller set of the form $\Omega_\e^{cr}$ for some $c$ depending only on $d$. We then have
\begin{eqnarray*}\nonumber
\sum_{k\in \e r \mathbb Z^d_{\rho\overline{\nu}}}\sum_{n=1}^d|\e r\rho|^d\frac{|u_\e^{\rho\overline{\nu}}(k+\nu_n)-u_\e^{\rho\overline{\nu}}(k)|^2}{|\e r\rho|^2} dx\ge
 \int_{\Omega_\e^{cr}}|\nabla u_\e^{\rho\overline{\nu}}|^2dx.
\end{eqnarray*}
From \eqref{Proof2} we then obtain
\begin{eqnarray}\label{Proof3}
&&\frac{1-s}{\e^{3-2s}}\lfloor u_\e\rfloor_s^2(\Omega_\e)
\ge \frac{r^{2(1-s)}}\e\frac{\sigma_{d}}{2d}\int_{(0,1)\times V}\int_{\Omega_\e^{cr}}|\nabla u_\e^{\rho\overline{\nu}}|^2dxd\mu_\e,
\end{eqnarray}
where $\mu_\e$ denotes the probability measure on $(0,1)\times V$ defined as follows
$$ d\mu_\e=
\frac{2(1-s)}{\mathcal{H}^{k_d}(V)}  \frac{1}{\rho^{2s-1}} d\mathcal{H}^{k_d}(\overline{\nu}) d\rho.
$$

{\bf Step 5.} {\em Estimate with a sequence in $H^1(\Omega_\e)$.}

\smallskip
We define the functions
$$
u^r_\e(x)=\int_{(0,1)\times V}u_\e^{\rho\overline{\nu}}(x) d\mu_\e(\rho,\overline\nu),
$$
where we have highlighted the dependence on $r$ of the interpolations, and note that
$$
\nabla u^r_\e(x)=\int_{(0,1)\times V}\nabla u_\e^{\rho\overline{\nu}}(x) d\mu_\e(\rho,\overline\nu)
$$
 in the sense of distributions.
An application of Jensen's inequality in \eqref{Proof3} then gives
\begin{eqnarray}\label{Proof31}
&&\frac{1-s}{\e^{3-2s}}\lfloor u_\e\rfloor_s^2(\Omega_\e)
\ge \frac{r^{2(1-s)}}\e\frac{\sigma_{d}}{2d}\int_{\Omega_\e^{cr}}|\nabla u^r_\e|^2dx.
\end{eqnarray}

{\bf Step 6.} {\em Compactness of comparison functions.}

\smallskip
We can apply compactness Lemma \ref{local-lemma} with $\omega$ replaced by any $\omega'$ compactly contained in $\omega$ and with $(0,1)$ replaced by $(cr,1-cr)$. We then obtain that the scaled functions $v_\e^r(x)=u_\e^r(x',\e x_d)$, up to subsequences and addition of constants, converge in $L^2(\omega'\times(cr,1-cr))$ to a function $v^r(x)=u^r(x')$ with $u^r\in H^1(\omega')$.

\medskip
{\bf Step 7.} {\em Compactness of the original sequence.}

\smallskip
We have
\begin{equation}\label{claim-1}
\frac1\e\int_{\omega'\times(\e cr,\e(1-cr))} |u_\e-u^r_\e|dx\le I^1_\e+I^2_\e,
\end{equation}
where
\begin{eqnarray*}
   &&I^1_\e:= \frac1\e \int_{[0,1]\times V} \sum_{{k \in {\mathring{\mathcal I}}^{\e r}_{\rho\overline{\nu}}(\Omega_\e)}} \int_{k+{\e r}Q_{\rho\overline{\nu}}} |u_\e(x) - u^{\rho\overline{\nu}}_\e(k)|\,dx\,d\mu_\e(\rho,\overline{\nu}) \\ \label{L1conv2}
&&I^2_\e:= \frac1\e \int_{[0,1]\times V} \sum_{{k \in {\mathring{\mathcal I}}^{\e r}_{\rho\overline{\nu}}(\Omega_\e)}} \int_{k+{\e r}Q_{\rho\overline{\nu}}} |u^{\rho\overline{\nu}}_\e(x) - u^{\rho\overline{\nu}}_\e(k)|\,dx\,d\mu_\e(\rho,\overline{\nu}),
\end{eqnarray*}
and
\[
{\mathring{\mathcal I}}^{\e r}_{\rho\overline{\nu}}(\Omega):=\{k\in \e r\mathbb Z_{\rho\overline{\nu}}^d : k+\e rc Q_{\rho\overline{\nu}} \subset \subset \Omega\}.
\]

To give a bound on $I^1_\e$ and $I_\e^2$, we can proceed as in \cite[Section 3.1]{BBD}, using the refined lower estimate in \eqref{Proof2}.

Using  a  scaled Poincar\'e--Wirtinger inequality (see, e.g. 
\cite[Theorem 6.33]{leofrac}), we have that
\begin{eqnarray*}\label{attempt4} && \hskip-1cm\nonumber
    \sum_{{k \in {\mathring{\mathcal I}}^{\e r}_{\rho\overline{\nu}}(\Omega_\e)}} \int_{k+{\e r}Q_{\rho\overline{\nu}}} |u_\e(x) - u_\e^{\rho\overline{\nu}}(k)|\,dx  \\ \nonumber 
    &\leq& P|{\e r}\rho|^{\frac{d}{2}+s} \sum_{k \in {\mathring{\mathcal I}}^{\e r}_{\rho\overline{\nu}}(\Omega_\e)}\bigg(\int_{(k+{\e r}Q_{\rho\overline{\nu}})^2}\frac{|u_\e(x)-u_\e(y)|^2}{|x-y|^{d+2s}}\,dxdy\bigg)^{1/2},
    \end{eqnarray*} 
where $P$ is the Poincar\'e--Wirtinger constant for the  $d$-dimensional unit cube. By this estimate, 
using the concavity of the square root and that $\#{\mathring{\mathcal I}}^{\e r}_{\rho\overline{\nu}}(\Omega_\e)\sim \frac{\e|\omega|}{\e^d\rho^{d}r^d}$, we then have 
\begin{eqnarray}\label{I1}
I^1_\e\le P \frac1\e \e^{s} r^s\,2^{1-\frac{d}{2}}\e^{\frac12}|\omega|^{\frac12}(1-s) \lfloor u_\e\rfloor_{H^{s}(\Omega)} \frac{1}{2-s}\ 
= P r^s\e\sqrt{1-s}\,2^{1-\frac{d}{2}}|\omega|^{\frac12} \frac{1}{2-s}  \sqrt{S}.
\end{eqnarray}

Regarding $I^2_\e$, we note that
$$
\frac1\e\int_{k+{\e r}Q_{\rho\overline{\nu}}} |u^{\rho\overline{\nu}}_\e(x) - u_\e^{\rho\overline{\nu}}(k)|\,dx\le r\rho \sqrt d\int_{k+rQ_{\rho\overline{\nu}}} |\nabla u^{\rho\overline{\nu}}_\e(x)|\,dx.
$$
This implies, using \eqref{Proof2}, that
\begin{eqnarray}\label{I2}
I^2_\e\le  r\sqrt\e \sqrt{\frac{d}{2^{d}}|\omega|\int_{[0,1]\times V} \sum_{k \in {\mathring{\mathcal I}}^{\e r}_{\rho\overline{\nu}}(\Omega_\e)}\int_{k+\e r Q_{\rho\overline{\nu}}} |\nabla u_\e^{\rho\overline{\nu}}(x)|^2 dx\, d\mu_\e(\rho,\overline{\nu}) }
\le\e\,   r^{2s-1} \sqrt{\frac{d|\omega|}{2^d{c_d}}S}.
\end{eqnarray}
From \eqref{claim-1}, \eqref{I1}, and \eqref{I2} we obtain that $u_\e\to u^ r$ in $L^1$. In particular, we see that $u^r$ is independent of $ r$, and that the convergence of the corresponding $v_\e$ is in $L^1_{\rm loc}(\omega\times(0,1))$.
\end{proof}

\begin{remark}[`Sub-optimal' lower bound]\label{sub-rem}
\rm  Let $s=s_\e\to s_0>1/2$ and let $u_\e\to u$ in $L^1$, with $u\in H^1(\omega)$. Then from \eqref{Proof31} and by the arbitrariness of $\omega'$ we have
\begin{eqnarray}\label{Proof7}
&&\liminf_{\e\to 0}\frac{1-s}{\e^{3-2s}}\lfloor u_\e\rfloor_s^2(\Omega_\e)
\ge r^{2(1-s_0)}(1-2cr)\frac{\sigma_{d}}{2d}\int_{\omega}|\nabla' u|^2dx'
\end{eqnarray}
for all $r\in(0,1/2c)$, where $c$ is the dimensional constant in Step 4 of the proof of Theorem \ref{nlcth-1}. This lower bound will be improved in the case $s_0<1$.
\end{remark}

\subsubsection{A Bourgain--Brezis--Mironescu-type result for $s\to 1^-$ }\label{BBM1}
The lower bound obtained in the previous section turns out to be sharp in the case $s_\e\to 1$.
As a consequence, we have the following Bourgain--Brezis--Mironescu-type result.

\begin{theorem}[Dimension-reduction Gamma-limit for $s\to 1$]\label{Gammath} Let $s=s_\e\to 1^-$ as $\e\to 0$. 
Then we have
$$
\Gamma\hbox{-}\lim_{\e\to 0} \frac{1-s}{\e^{3-2s}}\lfloor u\rfloor_s^2(\Omega_\e)= \frac{\sigma_{d}}{2d}\int_\omega |\nabla'u|^2dx'
$$
for all $u\in H^1(\omega)$, where the $\Gamma$-limit is computed with respect to the dimension-reduction convergence in $L^1_{\rm loc}$.
\end{theorem}

\begin{proof} Let $u_\e\buildrel{\hbox {\rm\scriptsize DR }\, \, }\over{\longrightarrow} u$ with respect to the local-$L^1$ dimension-reduction convergence. From Remark \ref{sub-rem} we then obtain 
$$
\liminf_{\e\to 0}\frac{1-s}{\e^{3-2s}}\lfloor u_\e\rfloor_s^2(\Omega_\e)
\ge (1-2cr)\frac{\sigma_{d}}{2d}\int_{\omega}|\nabla' u|^2dx',
$$
which gives the desired lower bound by letting $r\to0$.

As for the upper bound, by a  density argument it suffices to consider $u\in C^2(\overline\omega)$.
Let then $u\in C^2(\overline\omega)$, and with an abuse of notation, let $u$ also denote the function $u=u(x')$, independent of the $d$-th variable, which we view as an element of  $H^s(\Omega_\e)$. 
We first note that 
\begin{equation}\label{8}
\limsup_{\e\to 0} \frac{1-s}{\e^{3-2s}} \int_{\{(x,y)\in\Omega_\e: |x-y|>r\e\}}\frac{|u(x)-u(y)|^2}{|x-y|^{d+2s}} dx\,dy=0
\end{equation}
for all $r>0$.
Indeed, if $L$ is such that $|u(x)-u(y)|\le L|x-y|$, then,
noting that if $|x-y|\ge 2\e$ then $|x'-y'|>\e$, we have
\begin{eqnarray*}&&
\hskip-1.5cm\int_{\{(x,y)\in\Omega_\e\times\Omega_\e: |x-y|>r\e\}}\frac{|u(x)-u(y)|^2}{|x-y|^{d+2s}} dx\,dy\\
&\le& \int_{\{(x,y)\in\Omega_\e\times\Omega_\e: |x-y|>r\e\}}L^2|x-y|^{2-d-2s} dx\,dy\\
&\le& C\Big(\int_{\Omega_\e} \int_{\{r\e<|\xi|<2\e\}}|\xi|^{2-d-2s} d\xi\,dy\\
&&+
\int_{\Omega_\e}\int_{\{y\in\Omega_\e: |x'-y'|>\e\}}|x'-y'|^{2-d-2s} dx\,dy\Big)\\\\
&\le& C\Big( \mathcal H^{d-1}(S^{d-1})\e|\omega|\Big( \int_{r\e}^{2\e}t^{1-2s} dt\Big)^+ +
\mathcal H^{d-2}(S^{d-2})\e^2 |\omega| \int_{2\e}^{+\infty}t^{-2s}dt\Big)
\\
&\le& C\Big(\frac{\e}{2(1-s)}((r\e)^{2-2s}-(2\e)^{2-2s})^++\frac{\e^2}{2s-1}(2\e)^{1-2s}\Big)\\
&\le& C\Big( \frac{\e^{3-2s}}{1-s}(r^{2-2s}-2^{2-2s})^++ \e^{3-2s}\Big).
\end{eqnarray*}
Note that we have implicitly used the fact that $s>1/2$.
Hence, we have
\begin{eqnarray}\label{a1}\nonumber
&&\hskip-2cm\frac{1-s}{\e^{3-2s}} \int_{\{(x,y)\in\Omega_\e: |x-y|>r\e\}}\frac{|u(x)-u(y)|^2}{|x-y|^{d+2s}} dx\,dy\\
&\le& C\big((r^{2-2s}-2^{2-2s})^++ 1-s\big).
\end{eqnarray}
Letting $s\to 1$, we have \eqref{8}.

From \eqref{8}, we obtain that the asymptotic behaviour of $\frac{1}{\e^{3-2s}}F_{\e,s}(u)$ is the same as that
of
$$
\frac{1-s}{\e^{3-2s}} \int_{\{(x,y)\in\Omega_\e: |x-y|<r\e\}}\frac{|u(x)-u(y)|^2}{|x-y|^{d+2s}} dx\,dy
$$
with truncated range of interactions.

We now simplify the asymptotic analysis when $|x-y|<r\e$. We can write
$$
u(x)-u(y)= \langle \nabla u(x), x-y\rangle + O(|x-y|^2)
$$
uniformly in $x$, so that, with fixed $\eta>0$
\begin{eqnarray*}
\big||u(x)-u(y)|^2- |\langle \nabla u(x), x-y\rangle|^2\big|
&\le& \eta|\langle \nabla u(x), x-y\rangle|^2 +C_\eta |x-y|^4\\
&\le& \eta C|x-y|^2 +C_\eta |x-y|^4,
\end{eqnarray*}
and
\begin{eqnarray}\label{a2}\nonumber
&&\hskip-1.5cm\bigg|\int_{\{(x,y)\in\Omega_\e: |x-y|<r\e\}}\frac{|u(x)-u(y)|^2}{|x-y|^{d+2s}} dx\,dy
\\ \nonumber
&& -\int_{\{(x,y)\in\Omega_\e: |x-y|<r\e\}}\frac{|\langle \nabla u(x), x-y\rangle|^2}{|x-y|^{d+2s}} dx\,dy\bigg|\\ \nonumber
&\le&  \eta C\e|\omega|\int_{|\xi|<r\e\}}{|\xi|^{2-d-2s}} d\xi+C_\eta \e|\omega|\int_{|\xi|<r\e\}}{|\xi|^{4-d-2s}} dx\,dy\\\nonumber
&\le&  C\Big( \eta \frac{1}{1-s}\e^{3-2s}+C_\eta \e^{5-2s}\Big)\\
&=&  C\frac{1}{1-s}\e^{3-2s}\Big(  \eta+ C_\eta (1-s)\e^2\Big).
\end{eqnarray}
Letting $\e\to 0^+$  first, by the arbitrariness of $\eta$ and this estimate, together with \eqref{8}, we also have that the asymptotic behaviour of $\frac{1-s}{\e^{3-2s}}\lfloor u\rfloor_s^2(\Omega_\e)$ is the same as that
of
$$
\frac{1-s}{\e^{3-2s}} \int_{\{(x,y)\in\Omega_\e: |x-y|<r\e\}}\frac{|\langle \nabla u(x), x-y\rangle|^2}{|x-y|^{d+2s}} dx\,dy.
$$

We now take $r<\frac12$, so that 
\begin{eqnarray*}
&&\hskip-1.5cm\int_{\{(x,y)\in\Omega_\e: |x-y|<r\e\}}\frac{|\langle \nabla u(x), x-y\rangle|^2}{|x-y|^{d+2s}} dx\,dy\\
&\ge &\int_{\omega\times (r\e,(1-r)\e)}\int_{B_{r\e}(x)}\frac{|\langle \nabla u(x), x-y\rangle|^2}{|x-y|^{d+2s}}\,dy\,dx\\
&=&\int_{\omega\times (r\e,(1-r)\e)}\int_{B_{r\e}}\frac{|\langle \nabla u(x), \xi\rangle|^2}{|\xi|^{d+2s}}\,d\xi\,dx\\
&=&\int_{\omega\times (r\e,(1-r)\e)}|\nabla u(x)|^2\frac1d\int_{B_{r\e}}|\xi|^{2-d-2s}\,d\xi\,dx\\
&=&\int_{\omega\times (r\e,(1-r)\e)}|\nabla u(x)|^2c_d(r\e)^{2-2s}\,dx\\
&=&(1-2r)r^{2-2s}\frac{\e^{3-2s}}{1-s}\frac{\sigma_{d}}{2d}\int_{\omega}|\nabla' u(x')|^2dx'.
\end{eqnarray*}
Between the third and fourth line of the previous formula, we have used the remark that, by the symmetry of the domain of integration, we have 
\begin{eqnarray*}
\int_{B_{r\e}}\frac{|\langle \nabla u(x), \xi\rangle|^2}{|\xi|^{d+2s}}\,d\xi=|\nabla u(x)|^2 \int_{B_{r\e}}\frac{|\langle e_j, \xi\rangle|^2}{|\xi|^{d+2s}}\,d\xi
=|\nabla u(x)|^2 \frac1d\int_{B_{r\e}}\frac{|\xi|^2}{|\xi|^{d+2s}}\,d\xi
\end{eqnarray*}
for all elements of the canonical basis $\{e_1,\ldots, e_d\}$.
Hence,
\begin{equation}\label{9}
\liminf_{\e\to 0}\frac{1-s}{\e^{3-2s}}\lfloor u\rfloor_s^2(\Omega_\e)\ge (1-2r)\frac{\sigma_{d}}{2d}\int_{\omega}|\nabla' u(x')|^2dx'
\end{equation}
for all $r<\frac12$. Conversely, for all $r>0$ we have
\begin{eqnarray}\label{a3}\nonumber
&&\hskip-1.5cm\int_{\{(x,y)\in\Omega_\e: |x-y|<r\e\}}\frac{|\langle \nabla u(x), x-y\rangle|^2}{|x-y|^{d+2s}} dx\,dy\\ \nonumber
&\le &\int_{\omega\times (0,\e)}\int_{B_{r\e}(x)}\frac{|\langle \nabla u(x), x-y\rangle|^2}{|x-y|^{d+2s}}\,dy\,dx
=r^{2-2s}\frac{\e^{3-2s}}{1-s}\frac{\sigma_{d}}{2d}\int_{\omega}|\nabla' u(x')|^2dx',
\end{eqnarray}
repeating the same computations as above, so that 
\begin{equation}\label{10}
\limsup_{\e\to 0}\frac{1}{\e^{3-2s}}F_{\e,s}(u)\le \frac{\sigma_{d}}{2d}\int_{\omega}|\nabla' u(x')|^2dx'.
\end{equation}
The desired limsup ineuality follows from \eqref{9} and \eqref{10} by letting $r\to 0$.\end{proof}

\subsubsection{Convergence for $s_0\in(1/2,1)$}
We complete our computations of $\Gamma$-limits in the case $s_0>1/2$.

\begin{theorem}[Dimension-reduction Gamma-limit in the super-critical regime]\label{Gammaths0} Let $s=s_\e\to s_0\in(1/2,1)$ as $\e\to 0$. 
Then we have
$$
\Gamma\hbox{-}\lim_{\e\to 0} \frac{1}{\e^{3-2s_\e}}\lfloor u\rfloor_{s_\e}^2(\Omega_\e)= \frac1{1-s_0} K_{s_0,d}\int_\omega |\nabla'u|^2dx'
$$
for all $u\in H^1(\omega)$, where the $\Gamma$-limit is computed with respect to the dimension-reduction convergence in $L^1_{\rm loc}$, and
\begin{equation}
K_{s,d}=  \frac{1}{(3-2s)(d-1)} \int_{
 \mathbb R^{d-1}}\frac{|\xi'|^2}{(1+|\xi'|^2)^{\frac d2+s}}d\xi'.
\end{equation}
\end{theorem}

The form of the coefficient $K_{s,d}$ highlights a combination of planar interactions, which give the integral in $\mathbb R^{d-1}$, and of interactions in the vertical directions weighted by $s$, whose integral gives the coefficient $1/(3-2s)$.

\begin{remark}[Comparison with $C_{s,d}$]\label{Ksd}\rm 
The coefficients $C_{s,d}$ and $K_{s,d}$ are related as follows \begin{equation}\label{CK}
K_{s,d}=\frac{1}{(2s-1)(3-2s)}C_{s,d}.
\end{equation}
For the sake of readability, the proof is postponed to the Appendix.
\end{remark}

\begin{proof}[Proof of Theorem \rm\ref{Gammaths0}] We use the notation
$$
F^s_\e(u)=\frac{1}{\e^{3-2s}}\lfloor u\rfloor_{s}^2(\Omega_\e)
$$
{\em Lower bound.}
Let $u_\e\buildrel{\hbox {\rm\scriptsize DR }\, \, }\over{\longrightarrow} u\in H^1(\omega)$ in $L^1_{\rm loc}$. Up to restricting the lower bound to a slightly smaller thin film, we can suppose that the convergence of the functions $v_\e$, where $v_\e(x',x_d)= u_\e(x', \e x_d)$, to $u$ is in $L^2(\omega\times (0,\e))$. With fixed $K>0$ for each $\omega'$ compact subset of $\omega$ we can estimate
\begin{eqnarray}
\label{newest}
F^s_\e(u_\e)\ge\frac{1}{\e^{3-2s}}\int_{B_{K\e}}\int_{\omega'}
\int_0^\e\int_0^\e \frac{|u_\e(x'+\xi',y_d)-u_\e(x',x_d)|^2}{((x_d-y_d)^2+|\xi|^2)^{\frac d2+s}}dx' dx_d dy_d d\xi'
\end{eqnarray}
for $\e$ small enough.
For every $\xi'\neq 0$ we introduce the probability measures
$$
d\mu_\e^{\xi'}=\frac{1}{C_\e(\xi') ((x_d-y_d)^2+|\xi'|^2)^{\frac d2+s}}dx_d dy_d,
$$
where
\begin{equation}\label{defCexi}
C_\e(\xi')=\int_0^\e\int_0^\e\frac{1}{((x_d-y_d)^2+|\xi'|^2)^{\frac d2+s}}dx_d dy_d,
\end{equation}
and the averaged functions
$$
u^{\xi'}_\e(z)=\int_0^\e\int_0^\e u_\e(z,x_d)d\mu_\e(x_d,y_d)=\int_0^\e\int_0^\e u_\e(z,y_d)d\mu_\e(x_d,y_d).
$$
Note that we have $u^{\xi'}_\e\to u$ in $L^2(\omega)$. From Jensen's inequality we then deduce that
\begin{eqnarray*}
F^s_\e(u_\e)\ge\frac{1}{\e^{3-2s}}\int_{B_{K\e}}|\xi'|^2 C_\e(\xi') \int_{\omega'}
\frac{|u^{\xi'}_\e(x'+\xi')-u^{\xi'}_\e(x')|^2}{|\xi'|^2} dx' d\xi'.
\end{eqnarray*}
We use the notation $\overline u_\e^{\xi'}$ for the piecewise-affine interpolations on lattices aligned with the vector $\xi'$ (in $\mathbb R^{d-1}$ instead of $\mathbb R^d$), we then obtain 
\begin{eqnarray}\label{corpo}
F^s_\e(u_\e)\ge\frac{1}{\e^{3-2s}(d-1)}\int_{B_{K\e}}|\xi'|^2 C_\e(\xi') \int_{\omega''}
|\nabla'\overline u_\e^{\xi'}|^2 dx' d\xi',
\end{eqnarray}
up to restricting to a slightly smaller $\omega''$. Up to a further average, analogous to that in Step 5 of the proof of Theorem \ref{nlcth-1}, we can suppose that $\overline u_\e^{\xi'}\to u$ weakly in $H^1(\omega'')$, 
so that 
\begin{eqnarray}\label{newest-2}
\liminf_{\e\to 0}F^s_\e(u_\e)\ge\liminf_{\e\to 0}\frac{1}{\e^{3-2s}(d-1)}\int_{B_{K\e}}|\xi'|^2 C_\e(\xi') d\xi'
\int_{\omega''} |\nabla' u|^2 dx' .
\end{eqnarray}
By a change of variable and Fubini's Theorem, we can compute
\begin{eqnarray*}
\int_{B_{K\e}}|\xi'|^2 C_\e(\xi') d\xi'
&=&\int_0^\e\int_0^\e\int_{B_{K\e}}\frac{|\xi'|^2}{((x_d-y_d)^2+|\xi'|^2)^{\frac d2+s}}d\xi'dx_d dy_d
\\
&=&\e^{3-2s}\int_0^1\int_0^1|t-\tau|^{1-2s}\int_{B_{\frac{K}{|t-\tau|}}}\frac{|\xi^\prime|^2}{(1+|\xi^\prime|^2)^{\frac d2+s}}d\xi^\prime dt d\tau. 
\end{eqnarray*}
Noting that, by Lebesgue Dominated Convergence Theorem,  
\begin{eqnarray*}
&&\hskip-1cm\lim_{K\to+\infty}\int_0^1\int_0^1|t-\tau|^{1-2s}\int_{B_{\frac{K}{|t-\tau|}}}\frac{|\xi^\prime|^2}{(1+|\xi^\prime|^2)^{\frac d2+s}}d\xi^\prime dt d\tau\\
&&=\int_0^1\int_0^1|t-\tau|^{1-2s} dt d\tau\int_{\mathbb R^{d-1}}\frac{|\xi^\prime|^2}{(1+|\xi^\prime|^2)^{\frac d2+s}}d\xi^\prime \\
&&=
\frac1{(1-s)(3-2s)}\int_{\mathbb R^{d-1}}\frac{|\xi^\prime|^2}{(1+|\xi^\prime|^2)^{\frac d2+s}}d\xi^\prime, 
\end{eqnarray*}
we obtain
\begin{equation}\label{limK}
\lim_{K\to+\infty} \int_{B_{K\e}}|\xi'|^2 C_\e(\xi') d\xi'
=\frac{\e^{3-2s}}{(1-s)(3-2s)}\int_{\mathbb R^{d-1}}\frac{|\xi^\prime|^2}{(1+|\xi^\prime|^2)^{\frac d2+s}}d\xi^\prime. 
\end{equation}
Eventually, we then have the estimate 
$$
\liminf_{\e\to 0}F^s_\e(u_\e)\ge \frac{1}{1-s_0}K_{s_0,d}\int_{\omega''} |\nabla' u|^2 dx',
$$
and we can finally let $\omega''$ invade $\omega$. 
\smallskip

{\em Pointwise convergence and upper bound.}
We show a pointwise convergence result for $u\in C^2(\mathbb R^{d-1})$. 

Note that, if we fix $K>0$ arbitrary, then the contribution of the integral in the set $\{|x^\prime-y^\prime|\geq \e K\}$ is negligible as $\e\to 0$ with respect to the contribution in the complementary set. Indeed,
\begin{eqnarray*}
\int_{\{|x^\prime-y^\prime|\geq \e K\}} \frac{|u(x^\prime)-u(y^\prime)|^2}{|x-y|^{d+2s}}\, dx\, dy\leq 
L^2|\omega|\int_{\{\e K\leq |\xi^\prime|\}} C_\e(\xi^\prime)|\xi^\prime|^2 d\xi^\prime, 
\end{eqnarray*}
where $L$ is a Lipschitz constant for $u$ in $\omega$, and $C_\e(\xi')$ is defined in \eqref{defCexi}. 
Since $C_\e(\xi')\leq \e^{2}|\xi^\prime|^{-d-2s}$ and $s>\frac{1}{2}$, we obtain  
\begin{eqnarray*}
\frac{1}{\e^{3-2s}}\int_{\{|x^\prime-y^\prime|\geq \e K\}} \frac{|u(x^\prime)-u(y^\prime)|^2}{|x-y|^{d+2s}}\, dx\, dy &\leq &C\frac{1}{\e^{1-2s}}
\int_{\{|x^\prime-y^\prime|\geq \e K\}} |\xi^\prime|^{2-d-2s} d\xi^\prime\\
&=&C\frac{1}{\e^{1-2s}} \sigma_{d-1}\int_{\e K}^{+\infty} t^{-2s}\, dt\\
&=&C\frac{1}{2s-1} \sigma_{d-1}K^{1-2s}, 
\end{eqnarray*} 
where $C>0$ does not depend on $K$, $\e$ and $s$. 
Since $K$ can be chosen arbitrarily large and $1-2s<0$, we obtain that we can consider the $s$-seminorm only in the set $\{|x^\prime-y^\prime|<\e K\}$. 
Then, we can estimate
\begin{eqnarray*}
{F_\e^{s}(u)}\leq \frac{1}{\e^{3-2s}}\int_{B_{K\e}}C_\e(\xi^\prime)\int_{\omega}
|u(x'+\xi')-u(x')|^2 dx^\prime d\xi^\prime+o(1)_{K\to+\infty}, 
\end{eqnarray*}
obtaining by symmetry 
\begin{equation}\label{up}
{F_\e^{s}(u)}\leq 
\frac{1}{\e^{3-2s}(d-1)}\int_{B_{K\e}}|\xi'|^2 C_\e(\xi') d\xi^\prime \int_{\omega}
|\nabla u|^2 dx^\prime +o(1)_{K\to+\infty}.
\end{equation}
Since \eqref{limK} holds, the upper bound follows
by letting $K\to+\infty$ in \eqref{up}, and recalling the lower-bound estimate we obtain the pointwise convergence of $F_\e^s(u)$. 
\end{proof}

\begin{remark}[Continuity as $s_0\to 1$]\label{soo1}\rm
We note that, after scaling by $1-s_0$, the limit of the $\Gamma$-limits obtained in Theorem \ref{Gammaths0} as $s_0\to 1^-$ converge to the $\Gamma$-limit obtained in 
Theorem \ref{Gammath}. Since the functionals are equicoercive  this can be deduced by some general topological arguments (see \cite{DM}), but, since we have an explicit formula for the coefficients, it suffices to check that 
\begin{equation}\label{sot1}
\lim_{s\to 1^-}K_{s,d}=\frac{\sigma_d}{2d}.
\end{equation}
For the sake of readability, we postpone this computation to the Appendix.
\end{remark}

\subsection{The critical scaling $s=1/2$}
The behaviour at $s=1/2$ can be described by refining the case $s_\e>1/2$, which must be adjusted with a logarithmic correction. The corresponding result is the following.

\begin{theorem}[Dimension-reduction Gamma-limit in the critical regime]\label{Gammaths1}
Let $\{u_\e\}$ be such that 
$$
\sup_{\e}\frac{1}{\e^{2}|\log\e|}\lfloor u_\e\rfloor_{1/2}^2(\Omega_\e)<+\infty.
$$
Then, up to subsequences and addition of constants, $\{u_
\e\}$ converges to $u\in H^1(\omega)$ with respect to the dimension-reduction convergence in $L^2$-weak.
Furthermore, we have
$$
\Gamma\hbox{-}\lim_{\e\to 0} \frac{1}{\e^{2}|\log\e|}\lfloor u\rfloor_{1/2}^2(\Omega_\e)= \frac{
\sigma_{d-1}
}{d-1}\int_\omega |\nabla'u|^2dx'
$$
for all $u\in H^1(\omega)$, where the $\Gamma$-limit is computed with respect to the dimension-reduction convergence in $L^2$-weak.
\end{theorem}

\begin{proof} We now use the notation
$$
F^{1/2}_\e(u)=\frac{1}{\e^{2}|\log\e|}\lfloor u\rfloor_{1/2}^2(\Omega_\e)
=\frac{1}{\e^{2}|\log\e|}\int_{\Omega_\e\times\Omega_\e}\frac{(u(x)-u(y))^2}{|x-y|^{d+1}}dx dy.
$$

In order to prove the equi-coerciveness of $F^s_\e$ and the lower bound we follow the line of the proof of Theorem \ref{Gammaths0}. In this case, though, the relevant interactions are between the scale $\e$ and the scale $1$.
We then modify the argument in the proof of Theorem \ref{Gammaths0} accordingly.  In particular, inequality \eqref{newest} can be changed to 
\begin{eqnarray}
\label{newest-1}
F^{1/2}_\e(u_\e)\ge\frac{1}{\e^{2}
|\log\e|}\int_{B_\delta\setminus B_{K\e}}\int_{\omega'}
\int_0^\e\int_0^\e \frac{|u_\e(x'+\xi',y_d)-u_\e(x',x_d)|^2}{((x_d-y_d)^2+|\xi|^2)^{\frac{d+1}2}}dx' dx_d dy_d d\xi'
\end{eqnarray}
for $\e$ small enough, where $\delta>0$ is any fixed positive number not exceeding the distance of $\omega'$ from $\partial\omega$. 
The arguments following \eqref{newest} in the proof of Theorem \ref{Gammaths0} can then be repeated word for word, arguing that, up to subsequences and addition of constants we can suppose that $u_\e\to u$ with $u\in H^1(\omega'')$. Correspondingly, inequality \eqref{newest-2} becomes
\begin{eqnarray}\label{newest-3}
\liminf_{\e\to 0}F^{1/2}_\e(u_\e)\ge\liminf_{\e\to 0}\frac{1}{\e^{2}|\log\e|(d-1)}C_{\e,K,\delta}
\int_{\omega''} |\nabla' u|^2 dx',
\end{eqnarray}
where
$$C_{\e,K,\delta}=
\int_{B_\delta\setminus B_{K\e}}\int_0^\e\int_0^\e\frac{|\xi'|^2 }{((x_d-y_d)^2+|\xi'|^2)^{\frac{d+1}2}}dx_d dy_d d\xi'$$
in analogy with \eqref{defCexi}.
Changing variables, we obtain 
$$
C_{\e,K,\delta}=\e^2\int_0^1\int_0^1\int_{B_{\delta/(\e|x_d-y_d|)}\setminus B_{K/|x_d-y_d|}}\frac{|\xi'|^2 }{(1+|\xi'|^2)^{\frac{d+1}2}}d\xi'dx_d dy_d .
$$
For fixed $t\neq 0$, we have
$$
\lim_{\e\to 0}\frac{\displaystyle\int_{B_{\delta/(\e t)}\setminus B_{K/t}}\frac{|\xi'|^2 }{(1+|\xi'|^2)^{\frac{d+1}2}}d\xi'}{\displaystyle\int_{B_{\delta/(\e t)}\setminus B_{K/t}}\frac{1 }{|\xi'|^{d-1}}d\xi'}=1;
$$
hence, in \eqref{newest-3} we can substitute $C_{\e,K,\delta}$ with 
$$
\e^2\int_0^1\int_0^1\int_{B_{\delta/(\e|x_d-y_d|)}\setminus B_{K/|x_d-y_d|}}\frac{1 }{|\xi'|^{d-1}}d\xi'd\xi'dx_d dy_d=\sigma_{d-1}(S^{d-2})\e^2
(\log\delta-\log\e-\log K).
$$
Using this expression in \eqref{newest-3}, we obtain 
$$\liminf_{\e\to 0} F^{1/2}_\e(u_\e)\ge\liminf_{\e\to 0}\frac{
\sigma_{d-1}
}{d-1}
\int_{\omega''} |\nabla' u|^2 dx'.
$$
The liminf inequality then follows by letting $\omega''$ tend to $\omega$.

\smallskip
In order to prove the upper bound, as usual, it suffices to show it for piecewise-affine functions and argue by density. Furthermore, since the $\Gamma$-limit is local, it suffices to construct a recovery sequence for an affine map $u$. Thanks to its Lipschitz continuity, for fixed $K>0$ we have
$$
\frac{1}{\e^{2}|\log\e|}\int_{(\Omega_\e\times\Omega_\e)\cap\{|x'-y'|\le K\e\}
}\frac{(u(x)-u(y))^2}{|x-y|^{d+1}}dx dy
\le C\frac1{|\log\e|}.
$$
Hence, we can limit ourselves to considering pairs with $|x'-y'|> K\e$, for which the computation is the same as that performed for the lower bound.
\end{proof}

\begin{remark}[asymptotic behaviour as $s_\e\to1/2$]\rm The asymptotic analysis described in the previous theorem can be extended to $s_\e\to 1/2$ as $\e\to 0$. In this case, the correct scaling depends on $s_\e$ and interpolates between the scaling $\e^2|\log\e|$ and the scalings $\e^2/(\frac12-s_\e)$ for $s_\e< 1/2$ and $\e^{3-2s_\e}/(s_\e-\frac12)$ for $s_\e>1/2$. We refer to \cite{Picerni} for a similar analysis in the context of fractional phase transitions (see also \cite{SV,Solci}).  
\end{remark}

\subsection{$\Gamma$-limit in the low-integrability case}
In the case $s_\e\to s_0<1/2$ we only have a dimension-reduction compactness result for the weak $L^2$ convergence. The corresponding $\Gamma$-limits are still given by lower-dimensional Gagliardo seminorms.

\subsubsection{A weak dimension-reduction compactness result}

We prove a dimension-reduction compactness result with respect to the weak $L^2$ convergence. It is based on the following general slicing result valid for all $s\in(0,1)$. The lemma is a version adapted to thin films of part of a slicing result in \cite[Lemma 6.35]{leofrac}, following the line of the proof therein. It provides an estimate only for the `vertical' part of the seminorm, a complete slicing result not being available. Note that the hypothesis that $\omega$ being bounded is not necessary, but some regularity of its boundary is used.

\begin{lemma}[Slicing on the `thin' direction]\label{slixd} There exists $C>0$ such that for all $\e>0$, $s\in(0,1)$, and $u\in H^s(\Omega_\e)$ we have
\begin{equation}\label{slixd-eq1}
\int_\omega\int_0^\e\int_0^\e\frac{|u(x^\prime,x_d)-u(x^\prime,y_d)|^2}{|x_d-y_d|^{1+2s}}\, dx_d\, dy_d\, dx^\prime\leq C \int_{\Omega_\e}\int_{\Omega_\e} \frac{|u(x)-u(y)|^2}{|x-y|^{d+2s}}\, dx\, dy. 
\end{equation}
\end{lemma}
\begin{proof}  
Given $x^\prime\in\omega$ and $x_d\neq y_d\in(0,\e)$, we consider the ball $B=B(x^\prime,x_d,y_d)$ centered at $z^\ast=(x^\prime,\frac{x_d+y_d}{2})$ with radius $\frac{|x_d-y_d|}{4}$. 
Note that, for $\e$ small enough (with respect to the dimensions of $\omega$),  
$|B\cap\Omega_\e|\geq c_\omega|B|$ with $c_\omega>0$ by the Lipschitz assumption on $\omega$. Hence, 
$$\frac{|u(x^\prime,x_d)-u(x^\prime,y_d)|^2}{|x_d-y_d|^{1+2s}}\leq \frac{2}{c_\omega|B|}\int_{B\cap\Omega_\e} \frac{|u(x^\prime,x_d)-u(z)|^2+|u(z)-u(x^\prime,y_d)|^2}{|x_d-y_d|^{1+2s}}\, dz.$$ 
By Fubini's Theorem and by exchanging the role of $x_d$ and $y_d$ in second term of the sum, we then get 
\begin{eqnarray}\label{sli-part}
&&\hspace{-1cm}\int_\omega\int_0^\e\int_0^\e\frac{|u(x^\prime,x_d)-u(x^\prime,y_d)|^2}{|x_d-y_d|^{1+2s}}\, dx_d\, dy_d\, dx^\prime\nonumber\\
&&\leq 
C\int_{\Omega_\e}\int_{\Omega_\e}|u(x)-u(z)|^2 \int_0^\e\chi_{B\cap\Omega_\e}(z) \frac{1}{|x_d-y_d|^{d+1+2s}}\, dy_d\, dx\, dz,
\end{eqnarray}
with $C>0$ only depending on $c_\omega$ and the dimension $d$. 
If $\chi_{B\cap\Omega_\e}(z)=1$, then $|z-z^\ast|\leq \frac{|x_d-y_d|}{4}$, and we can estimate 
$$|x-z|\leq |z-z^\ast|+|x-z^\ast|\leq \frac{|x_d-y_d|}{4} +\frac{|x_d-y_d|}{2};$$ 
We obtain 
\begin{eqnarray*}
\int_0^\e\chi_{B\cap\Omega_\e}(z) \frac{1}{|x_d-y_d|^{d+1+2s}}\, dy_d \leq 2\int_{\frac{4|x-z|}{3}}^{+\infty} t^{-d-1-2s}\, dt=\frac{2}{d+2s}\Big(\frac{3}{4}\Big)^{d+2s} \frac{1}{|x-z|^{d+2s}}.
\end{eqnarray*}
From \eqref{sli-part}, the claim follows with $C>0$ depending only on $\omega$ and the dimension $d$. 
\end{proof} 
Note that, in terms of the scaled function $v\colon\omega\times (0,1)\to \mathbb R$ given by 
$v(x^\prime, t)=u(x^\prime,\e t)$, estimate \eqref{slixd-eq1} becomes 
\begin{equation}\label{slixd-eq}
\int_\omega\int_0^1\int_0^1\frac{|v(x^\prime,x_d)-v(x^\prime,y_d)|^2}{|x_d-y_d|^{1+2s}}\, dx_d\, dy_d\, dx^\prime \leq C\e^{2s-1}\int_{\Omega_\e}\int_{\Omega_\e}\frac{|u(x)-u(y)|^2}{|x-y|^{d+2s}}  \, dx\, dy.
\end{equation}

\begin{remark}[Rigidity]\label{rigidity-rm}\rm 
Lemma \ref{slixd} implies a rigidity result for `pointwise' convergence.  
More precisely, given $u\colon\omega\times (0,1)\to\mathbb R$ and defining 
$u_\e(x^\prime,x_d)=u(x^\prime, \frac{x_d}{\e})$, 
by \eqref{slixd-eq} we obtain
$$\int_\omega\int_0^1\int_0^1 \frac{|u(x^\prime,x_d)-u(x^\prime,y_d)|^2}{|x_d-y_d|^{1+2s}}\, dx_d\, dy_d\, dx^\prime\leq C\e^{2s-1}\lfloor u_\e\rfloor_{s_\e}^2(\Omega_\e).$$ 
Then, if $\e^{2s-1}\lfloor u_\e\rfloor_{s_\e}^2(\Omega_\e)=o(1)_{\e\to 0}$, the function $u$ does not depend on the $d$-th variable; that is, $u(x^\prime,x_d)=v(x^\prime)$. 
\end{remark}

Lemma \ref{slixd} allows us to prove a compactness result, in the sense of the weak dimension-reduction convergence in $L^2$ (see Definition \ref{drco}). 
We start by showing a compactness result for sequences bounded in $L^2$. 

\begin{lemma}
\label{compactness1}  
Let $\{u_\e\}$ be a bounded sequence in $L^2(\omega\times(0,1))$ such that 
$$\lim_{\e\to 0}\frac{1}{\e^{2s_\e-1}} \lfloor u_\e\rfloor_{s_\e}^2(\Omega_\e)=0.$$ 
Then
there exists $u\in L^2(\omega)$ such that, up to the addition of constants and up to subsequences, 
$u_\e\buildrel{\hbox {\rm\scriptsize DR }\, \, }\over{\longrightarrow} u$ weakly in $L^2$ in the sense of Definition {\rm \ref{drco}.} 
\end{lemma}

\begin{proof}
Up to the addition of constants and up to subsequences, we have the weak convergence $v_\e\weak v$. 
To show that $v$ does not depend on the `thin' variable, we apply the Poincar\'e inequality and Lemma \ref{slixd}, obtaining 
\begin{eqnarray}\label{rigidity}
\int_\omega\int_0^1 |v_\e(x^\prime, x_d)- \overline v_\e(x^\prime)|^2\, dx^\prime\, dx_d&\leq&C \int_\omega \int_0^1\int_0^1 
\frac{|v_\e(x^\prime, x_d)-v_\e(x^\prime, y_d)|^2}{|x_d-y_d|^{1+2s_\e}}\, dx_d\, dy_d\, dx^\prime\nonumber \\
&\leq&C \e^{2s_\e-1} \lfloor u_\e\rfloor_{s_\e}^2(\Omega_\e)\ \leq \ C \e^{2s_\e+1}=o(1)_\e,
\end{eqnarray}
where $\overline v_\e(x^\prime)=\int_0^1v_\e(x^\prime,t)\, dt$. Since $\overline v_\e\weak \overline v$ in $L^2(\omega)$, 
where $\overline v(x^\prime)=\int_{0}^1v(x^\prime, t)\, dt$, by \eqref{rigidity} and the lower semicontinuity of the norm we find that the weak limit $v$ does not depend on $x_d$. 
\end{proof}

\begin{theorem}[Dimension-reduction compactness]\label{compactness2} Let $\{u_\e\}$ be such that 
$$\sup_\e \frac{1}{\e^2} \lfloor u_\e\rfloor_{s_\e}^2(\Omega_\e)<+\infty.$$ 
Then
there exists $u\in L^2(\omega)$ such that, up to the addition of constants and up to subsequences, 
$u_\e\buildrel{\hbox {\rm\scriptsize DR }\, \, }\over{\longrightarrow} u$ weakly in $L^2$ in the sense of Definition \rm \ref{drco}.  
\end{theorem}

%
%
%

\begin{proof}
By the Poincar\'e inequality (see \cite[Theorem 6.33]{leofrac}), we have
$$\|u_\e - \overline u_\e \|^2_{L^2(\Omega_\e)}\leq \frac{C}{\e}\lfloor u_\e\rfloor_{s_\e}^2(\Omega_\e), $$
where $\overline u_\e=\frac{1}{|\Omega_\e|}\int_{\Omega_\e} u_\e(x)\, dx$.  
Then
$$\|v_\e - \overline u_\e \|^2_{L^2(\omega\times(0,1))}\leq \frac{C}{\e^2}\lfloor u_\e\rfloor_{s_\e}^2(\Omega_\e),$$
which is bounded by assumption. The conclusion follows from Lemma \ref{compactness1}.
\end{proof}

\subsubsection{Increase of integrability in dimension reduction}

Lemma \ref{compactness2} suggests the scaling for a dimension reduction limit of the Gagliardo seminorms. This is confirmed by the following theorem, which shows a gain of $1/2$ in the Gagliardo seminorm.

\begin{theorem}[Dimension-reduction Gamma-limit in the sub-critical regime]\label{drconv}
Let $s=s_\e\to s_0<\frac12$; then the $\Gamma$-limit with respect to the dimension-reduction convergence in $L^2$ of 
\begin{equation}\label{scaGa}
F_\e^s(u)=\frac1{\e^2} \int_{\Omega_\e}\int_{\Omega_\e}\frac{(u(x)-u(y))^2}{
|x-y|^{d+2s}}dxdy
\end{equation}
is given on $L^2(\omega)$ by the $H^{\frac12+s_0}$-seminorm squared in $\omega$; that is, noting that
$d-1+2(\frac12+s_0)=d+2s_0$, by 
$$
F^{s_0}(u)=\int_\omega\int_\omega 
\frac{|u(x')-u(y')|^2}{|x'-y'|^{d+2s_0}}dx' dy'.
$$
\end{theorem}

\begin{proof}
{\em Lower bound.} 
By compactness, the limit is independent of $x_d$.
If $u_\e\to u$; that is, $v_\e\weak u$ in $L^2(\omega\times(0,1))$, then for all fixed $\delta>0$ and $s=s_\e\to 0$
\begin{eqnarray*}\liminf_{\e\to 0}F_\e^s(u_\e)
&\ge& \liminf_{\e\to 0}\frac{1}{\e^2}\int_{(\Omega_\e\times \Omega_\e)\cap \{|x'-y'|>\delta\}}\frac{|u_\e(x)-u_\e(y)|^2}{|x-y|^{d+2s_\e}}dx dy
\\
&=&\liminf_{\e\to 0}\int_0^1\int_0^1\int_{(\omega\times\omega)\cap \{|x'-y'|>\delta\}}\frac{|v_\e(x)-v_\e(y)|^2}{|x'-y'|^{d+2s_\e}}dx dy
\\
&\ge&\int_0^1\int_0^1\int_{(\omega\times\omega)\cap \{|x'-y'|>\delta\}}\frac{|u(x')-u(y')|^2}{|x'-y'|^{d+2s_0}}dx dy
\end{eqnarray*}
The lower bound is optimized by letting $\delta\to0$.

\bigskip
{\em Upper bound.} For $u$ Lipschitz, note that
\begin{eqnarray*}
&&\hskip-1cm\frac{1}{\e^2}\int_{(\Omega_\e\times \Omega_\e)\cap \{|x'-y'|<\delta\}}\frac{|u(x)-u(y)|^2}{|x-y|^{d+2s_\e}}dx dy
\\
&&\le \frac{1}{\e^2}\Big(\int_{(\Omega_\e\times \Omega_\e)\cap \{|x-y|< 2\e\}}\frac{|u(x)-u(y)|^2}{|x-y|^{d+2s_\e}}dx dy
\\&&\qquad +
C\int_{(\Omega_\e\times \Omega_\e)\cap \{2\e<|x'-y'|<\delta\}}\frac{|u(x)-u(y)|^{2}}{|x'-y'|^{d+2s_\e}}dx dy\Big)
\\
&&\le \frac{C}{\e^2}\Big(\int_{(\Omega_\e\times \Omega_\e)\cap \{|x-y|< 2\e\}}\frac{1}{|x-y|^{d-2+2s_\e}}dx dy
\\&&\qquad +
\e^2\int_{(\omega\times\omega)\cap \{2\e<|x'-y'|<\delta\}}\frac{1}{|x'-y'|^{d-2+2s}}dx dy\Big)
\\
&&\le \frac{C}{\e^2}\Big(\e\int_0^{2\e}t^{1-2s_\e}dt
+
\e^2\int_{\e}^{\delta}t^{-2s_\e}dt\Big)
\\
&&\le C\big(\e^{1-2s_\e}
+
\delta^{1-2s_\e}\big),
\end{eqnarray*}
while
\begin{eqnarray*}
&&\lim_{\e\to 0}\frac{1}{\e^2}\int_{(\Omega_\e\times \Omega_\e)\cap \{|x'-y'|>\delta\}}\frac{|u(x')-u(y')|^2}{|x-y|^{d+2s_\e}}dx dy
\\ &&\qquad\qquad
=\int_{(\omega\times\omega)\cap \{|x'-y'|>\delta\}}\frac{|u(x')-u(y')|^2}{|x'-y'|^{d+2s_0}}dx dy.
\end{eqnarray*}
Hence, 
$$\limsup_{\e\to 0}F_\e^s(u)
\le\int_{\omega\times\omega}\frac{|u(x')-u(y')|^2}{|x'-y'|^{d+2s_0}}dx dy,
$$
by the arbitrariness of $\delta$.

For $u\in H^{\frac12+s_0}(\omega)$ the result is obtained by approximations.
\end{proof}

\begin{remark}[Beyond Gagliardo seminorms]\rm
Theorem \ref{drconv} holds more generally for $s_\e\to s_0\in(-\frac12, \frac12)$, with the same proof. If $s\le 0$, the integral in \eqref{scaGa} is not a Gagliardo seminorm. However, it is well defined for the functions in $L^2(\Omega)$. The resulting $\Gamma$-limit can still be interpreted as a squared $s_0+\frac12$ Gagliardo seminorm. 
\end{remark}


\subsection{Asymptotic behaviour for $s_0\to 1/2$}
We now analyze the behaviour of the $\Gamma$-limits obtained for $s_0\neq1/2$ at the point $1/2$. We recall that, for $s_0<1/2$ and $s_\e\to s_0$, we have 
$$
F^{s_0}(u)=\Gamma\hbox{-}\lim_{\e\to 0}\frac{1}{\e^{2}}\lfloor u\rfloor_{s_\e}^2(\Omega_\e)
=\int_\omega\int_\omega 
\frac{|u(x')-u(y')|^2}{|x'-y'|^{d+2s_0}}dx' dy'.
$$
For $s_0>1/2$ and $s_\e\to s_0$ we have 
$$
F^{s_0}(u)=\Gamma\hbox{-}\lim_{\e\to 0}\frac{1}{\e^{3-2s_\e}}\lfloor u\rfloor_{s_\e}^2(\Omega_\e)=\frac1{1-s_0}K_{s_0,d}\int_\omega |\nabla'u|^2dx'.
$$

\begin{proposition}[Blow-up at the critical regime]
 Let $F^{s_0}$ be defined as above for $s_0\neq1/2$. Then we have
 $$
 \lim_{s_0\to1/2}\Big|s_0-\frac12\Big|F^{s_0}(u)=\frac{
\sigma_{d-1}}{2(d-1)}\int_\omega |\nabla'u|^2dx'.
 $$
\end{proposition}

\begin{proof}
For the right-hand side limit we only have to compute the asymptotic behaviour of $K_{s_0,d}$.
For the sake of notational simplicity, we rename $s=s_0$,
and check that
$$
\lim_{s\to 1/2^+}(2s-1)K_{s,d}= \frac{
\sigma_{d-1}}{2(d-1)}. 
$$

For $s\to 1/2^+$ the behaviour of $K_{s,d}$ is the same as that of 
$$\frac{\sigma_{d-1}}{2(d-1)}\int_0^{+\infty}\frac{z^d}{(1+z^2)^{\frac d2+s}}dz
.$$
We directly estimate this integral. A less direct way would be to resort to its representation in terms of the function $\Gamma$.

Let $N\geq 1$ be fixed. 
Then, for all $z\geq N$ 
$$\Big(\frac{N^2}{N^2+1}\Big)^{\frac d2+s}z^{-2s} \leq \frac{z^d}{(1+z^2)^{\frac d2+s}}\leq z^{-2s},$$ 
and we estimate 
$$\Big(\frac{N^2}{N^2+1}\Big)^{\frac d2+s}\frac{N^{1-2s}}{2s-1} \leq \int_N^{+\infty}\frac{z^d}{(1+z^2)^{\frac d2+s}}\, dz \leq 
\frac{N^{1-2s}}{2s-1}.$$ 
Since $N$ is arbitrary, it follows that 
$$\lim_{s\to\frac{1}{2}+}(2s-1)\int_N^{+\infty}\frac{z^d}{(1+z^2)^{\frac d2+s}}\, dz =1.$$ 
Since the integral in $(0,N)$ is bounded by $N^{d+1}$ independently of $s$, we can conclude.

    If $s_0<1/2$ then $F^{s_0}(u)=\lfloor u   \rfloor^2_{s_0+\frac12}(
    \omega)$. From Theorem 
    \ref{BBM-thm} in the $d-1$ setting, we have 
    $$
    \lim_{s\to 1^-}(1-s)\lfloor u   \rfloor^2_{s}(
    \omega)=\frac{
\sigma_{d-1}}{2(d-1)}\int_\omega |\nabla'u|^2dx'.$$
We can conclude by  taking $s=s_0+\frac12$.
\end{proof}

\subsection{The case $\omega$ unbounded}
It is worth noting that for $\omega$ unbounded, the results do not change, even for $s_\e\to 0$, contrary to the case studied by Maz'ya and Shaposhnikova in \cite{MS} (see also \cite{MR4165063,MR4453966, MR4953682,pin25, MR4374610}).
This is due to the increase of differentiability in the dimension-reduction process. We briefly state and proof this fact in the following result.

\begin{theorem}
The limit dimension-reduction $\Gamma$-convergence theorems in the previous sections hold with respect to the corresponding dimension-reduction convergences of $u_{\e}$ to $u$ locally in $\omega$.
\end{theorem}

\begin{proof} The lower bound is achieved by the trivial estimate
$$
\lfloor u\rfloor^2_{s}(\Omega_\e)\ge
\lfloor u\rfloor^2_{s}((\omega\cap B_T)\times (0,\e)).
$$
We can then apply the corresponding theorem replacing $\omega$ with $\omega\cap B_T$, and let $T\to+\infty$.

Conversely, to prove the upper bound, by density, it suffices to consider functions $u
\in C^\infty_c(\omega)$. Let $T$ be such that the support of $u$ is contained in $B_T$. Since in this argument we may take $T$ arbitrarily large, we have 
\begin{eqnarray*}
\int_{B_T\times(0,\e)}\int_{(\mathbb R^{d-1}\setminus B_T)\times(0,\e)}\frac{(u(x)-u(y))^2}{
|x-y|^{d+2s}}dxdy
&\sim& \sigma_{d-1}\e^2 \int_{\omega}|u(x')|^2dx'\int_T^{+\infty}t^{-2-2s}dt\\
&=&\sigma_{d-1}\e^2 \int_{\omega}|u(x')|^2dx'\frac1{(1+2s)T^{1+2s}}.
\end{eqnarray*}
If $s<1/2$ this term is negligible by letting $T\to+\infty$. In the other cases, it is even negligible for a fixed $T$ as $\e\to 0$. This shows that the constant sequence $u_\e(x)=u(x')$ is a recovery sequence also in this case.
\end{proof}

\section{Development by $\Gamma$-convergence}
In this section, we re-read our results in terms of developments by $\Gamma$-convergence as in \cite{AnzB}. In that framework, in particular, given a sequence $F_\e$, we can write
$$
F_\e\ \buildrel\Gamma\over = \ \lambda_{1,\e} F^1+ \lambda_{2,\e} F^2+ o(\lambda_{2,\e})
$$
if $\lambda_{1,\e}>\!>\lambda_{2,\e}$,
$F^1_\e:=F_\e/\lambda_{1,\e}$ $\Gamma$-converge to $F^1$, $F^2_\e:=(F_\e-\min F^{1})/\lambda_{2,\e}$ $\Gamma$-converge to $F^2$,
and the $\Gamma$-limit of $\lambda_{1,\e}(F_\e-\min F^{1})/\eta_{\e}$ is $0$ in the set of minimizers of $F^{1}$ if $\lambda_{1,\e}>>\eta_\e>>\lambda_{2,\e}$. 
We suppose that the $\Gamma$-limits exist and are not trivial.
The $\Gamma$-limits are considered with respect to a suitable topology, but, in general, the coerciveness properties of the functionals $F^2_\e$ improve with respect to those of the functionals $F^1_\e$.
In the setting of \cite{AnzB}, functionals are considered as defined in a common space. In our case, this holds after scaling the variable $u$ to $v$, so that more generally we use the concept of expansion by $\Gamma$-convergence in the sense of \cite{BT}.

In order to clarify the use of the notation above, we first consider the local case $u\in H^1(\Omega_\e)$,
and consider the functionals 
$$
F_\e(u)=\int_{\Omega_\e}|\nabla u|^2dx.
$$
Using the variable $v$, scaling the $x_d$-variable, we obtain
$$
F_\e(u)=\frac1\e   \int_{\omega\times (0,1)}\Big|\frac{\partial{v}}{\partial x_d}\Big|^2dx
+\e \int_{\omega\times (0,1)}|\nabla' v|^2dx.
$$
We choose $
\lambda_{1,\e}=1/\e$
and the weak $L^2(\omega\times (0,1))$-topology for the scaled functions, with respect to which we compute the $\Gamma$-limit of
$$
F^1_\e(v)=\frac{F_\e(u)}{\lambda_{1,\e}}=\int_{\omega\times (0,1)}\Big|\frac{\partial{v}}{\partial x_d}\Big|^2dx
+\e^2 \int_{\omega\times (0,1)}|\nabla' v|^2dx;
$$
that is,
$$
F^1(v)= \int_{\omega\times (0,1)}\Big|\frac{\partial{v}}{\partial x_d}\Big|^2dx,
$$
with domain the functions $v\in L^2(\omega\times (0,1))$ whose weak derivative $\frac{\partial{v}}{\partial x_d}$ is also in $L^2(\omega\times (0,1))$. The minimum of $ F^1$ is $0$, and is achieved on functions depending only on $x'$. We then choose $\lambda_{2,\e}=\e$, and we compute the $\Gamma$-limit of 
 $$
F_\e^2(v)= \frac1{\e}F_\e(u)=\frac1{\e^2}   \int_{\omega\times (0,1)}\Big|\frac{\partial{v}}{\partial x_d}\Big|^2dx
+\int_{\omega\times (0,1)}|\nabla' v|^2dx,
$$
which is the functional finite only on functions $v\in H^1(\omega\times (0,1))$ depending only on $x'$, on which
$$
F^2(v)=\int_{\omega\times (0,1)}|\nabla' v|^2dx.
$$ 
Finally, if we take $\e<\!<\eta_\e<\!<\frac1\e$, we have
$$
\Gamma\hbox{-}
\lim_{\e\to 0} \Big(\frac1{\e\eta_\e}   \int_{\omega\times (0,1)}\Big|\frac{\partial{v}}{\partial x_d}\Big|^2dx
+\frac\e{\eta_\e} \int_{\omega\times (0,1)}|\nabla' v|^2dx\Big)=\begin{cases} 0 &\hbox{ if $\frac{\partial{v}}{\partial x_d}=0$}\cr 
+\infty&\hbox{ otherwise,}
\end{cases}
$$
by comparison with $F^1$ and $F^2$.

In this sense, we have 
$$
\int_{\Omega_\e}|\nabla u|^2dx\ \buildrel\Gamma\over =\ \frac1\e   \int_{\omega\times (0,1)}\Big|\frac{\partial{v}}{\partial x_d}\Big|^2dx +\e \int_{\omega}|\nabla' v|^2dx' +o(\e),
$$
using the dimension-reduction notation with apices in the second integral as a shorthand to highlight that its domain is the set of $H^1$-functions independent of $x_d$.

\smallskip
In light of the above explanation, we can state our results as follows (after noting that $C_{1/2,d}=\frac{\sigma_{d-1}}{d-1}$).

\begin{theorem}[$\Gamma$-expansions]\label{gammaexp}
Let $s_\e\in(0,1)$ with $s_\e\to s_0\neq1/2$. Then we have the following $\Gamma$-expansions with respect to the weak convergence of $v$ in $L^2(\omega\times(0,1))$:

{\rm(i)} if $s_0\in[0,1/2)$ 
\begin{eqnarray*}&&\hskip-3cm
\lfloor u\rfloor^2_{s_\e}(\Omega_\e)\ \buildrel\Gamma\over =\ \e^{1-2s_\e}C_{s_0,d}\int_{\omega}\int_0^1\int_0^1\frac{|v(x^\prime, x_d)-v(x^\prime, y_d)|^2}{|x_d-y_d|^{1+2s_0}}\, dx_d\, dy_d\, dx^\prime
\\
&& +\e^2 \int_\omega\int_\omega 
\frac{|v(x')-v(y')|^2}{|x'-y'|^{d+2s_0}}\, dx^\prime \, dy^\prime
+o(\e^2);    
\end{eqnarray*}

{\rm(ii)} if $s_0=1/2$, then 
\begin{eqnarray*}&&\hskip-3.6cm
\lfloor u\rfloor^2_{1/2}(\Omega_\e)\ \buildrel\Gamma\over =\frac{\sigma_{d-1}}{d-1}\int_{\omega}\int_0^1\int_0^1\frac{|v(x^\prime, x_d)-v(x^\prime, y_d)|^2}{|x_d-y_d|^{2}}\, dx_d\, dy_d\, dx^\prime
\\ &&
+\e^{2}|\log\e|\frac{\sigma_{d-1}}{d-1}\int_{\omega} |\nabla' v|^2 dx'
+o(\e^{2}|\log\e|);
\end{eqnarray*}

{\rm(iii)} if $s_0\in(1/2,1)$ 
\begin{eqnarray*}&&\hskip-3cm
\lfloor u\rfloor^2_{s_\e}(\Omega_\e)\ \buildrel\Gamma\over =\ \e^{1-2s_\e}C_{s_0,d}\int_{\omega}\int_0^1\int_0^1\frac{|v(x^\prime, x_d)-v(x^\prime, y_d)|^2}{|x_d-y_d|^{1+2s_0}}\, dx_d\, dy_d\, dx^\prime
\\ &&
+\e^{3-2s_\e}K_{s_0,d}\int_{\omega} |\nabla' v|^2 dx'
+o(\e^{3-2s_\e});
\end{eqnarray*}

{\rm(iv)} if $s_0=1$ 
$$
\lfloor u\rfloor^2_{s_\e}(\Omega_\e)\ \buildrel\Gamma\over =\ \frac{\e^{1-2s_\e}}{1-s_\e}\frac{\sigma_{d}}{2d}\int_{\omega}\int_0^1\Big|\frac{\partial v}{\partial x_d}\Big|^2 dx_d\, dx^\prime
+\frac{\e^{3-2s}}{1-s_\e} \frac{\sigma_{d}}{2d}\int_{\omega} |\nabla' v|^2 dx'
+o\Big(\frac{\e^{3-2s}}{1-s}\Big).
$$
\end{theorem}

\begin{remark}[Separation of scales]\rm
When $s_0=1$ it may be interesting to consider the case $
1-s_\e=O(\frac1{|\log\e|})$, for which we can assume that $\e^{1-s_\e}$ tends to $\kappa\in(0,1]$. In that case, the expansion can also be written as 
$$
(1-s_\e)\lfloor u\rfloor^2_{s_\e}(\Omega_\e)\ \buildrel\Gamma\over =\ \frac{\kappa^2}{
\e}\frac{\sigma_{d}}{2d}\int_{\omega}\int_0^1\Big|\frac{\partial v}{\partial x_d}\Big|^2 dx_d\, dx^\prime
+\
\e\kappa^2\frac{\sigma_{d}}{2d}\int_{\omega} |\nabla' v|^2 dx'
+o(\e)
$$

In particular, in the regime where $\e^{1-s_\e}\to \kappa>0$; that is, if
$1-s_\e=O(\frac1{|\log\e|})$,
we have $\e^{1-2s_\e}\sim{\kappa^2}/\e$ and $\e^{3-2s_\e}\sim {\kappa^2}\e$, so that 
$$
(1-s_\e)\lfloor u\rfloor^2_{s_\e}(\Omega_\e)\ \buildrel\Gamma\over =\ \frac1\e{\kappa^2}\frac{
\sigma_{d}}{2d}\int_{\omega}\int_0^1\Big|\frac{\partial v}{\partial x_d}\Big|^2 dx_d\, dx^\prime
+\e{\kappa^2}\frac{\sigma_{d}}{2d}\int_{\omega} |\nabla' v|^2 dx'
+o(\e).
$$
If $\kappa=1$; that is, if $1-s_\e=o(\frac1{|\log\e|})$, this result shows a {\em separation of scale effect}; that is, the asymptotic analysis is formally the same as letting first $s\to 1$ with $\e$ fixed, so that, by the results of 
Bourgain, Brezis, and Mironescu,
$$
\Gamma\hbox{-}\lim_{s\to 1} (1-s)\lfloor u\rfloor^2_{s}(\Omega_\e)= \frac{\sigma_{d}}{2d}\int_{\Omega_\e}|\nabla u|^2dx,
$$
and then using the expansion 
\begin{equation}\label{intro-4}
\int_{\Omega_\e} |\nabla u|^2dx\ \buildrel \Gamma\over =\ \frac1\e \int_{\omega\times (0,1)}\Big|\frac{\partial v}{\partial x_d}\Big|^2dx+ \e\int_{\omega}|\nabla'v|^2dx' +o(\e),
\end{equation} 
for the analysis as $\e\to 0$.
\end{remark}

\section{Behaviour of Gagliardo seminorms in $W^{s,p}$ for general $p$}
    We briefly comment on the case $p\neq 2$, for which we consider 
    \[\lfloor u
\rfloor^p_s(\Omega_\e)=\int_{\Omega_\e\times \Omega_\e} \frac{|u(x)-u(y)|^p}{|x-y|^{d+sp}} dx\,dy.\]
In the following, we state the convergence results, with a short explanation of the necessary changes in the statements and  proofs. The proofs are only slightly more complicated in the notation, but follow the same lines. 

\subsection{First scaling}

For general $p$, the first scaling is $\e^{1-sp}$. Arguing as in the derivation of the scaling for $p=2$ in Section \ref{asymp2}, we can write
\begin{equation*}
\lfloor u\rfloor_s^p(\Omega_\e)=\e^{d-sp}
\int_0^1\int_0^1(v(z_d)-v(w_d))^p\Big(\int_{\frac\omega\e\times\frac\omega\e}\frac{1}{|z-w|^{d+sp}}dz' dw'\Big)dz_d dw_d
\end{equation*}
and
\begin{equation*}
\int_{\frac\omega\e\times\frac\omega\e}\frac{1}{|z-w|^{d+sp}}dz' dw'
\sim \left(\frac{|\omega|}{\e^{d-1}}\int_{\mathbb R^{d-1}}\frac1{(1+ |\zeta|^2)^{\frac{d+sp}{2}}}d\zeta\right) \frac1{|z_d-w_d|^{1+sp}},
\end{equation*}
deducing that the natural scaling factor is $\e^{1-sp}$.

We can state the analog of Theorem \ref{fulldim}, with the constant $C_{s,d;p}$ defined by
\begin{equation}\label{Csdp}
C_{s,d;p}=\int_{\mathbb R^{d-1}}
\frac{1}{(1+|\xi|^2)^{\frac{d+sp}{2}}}\, d\xi
\end{equation} 
in place of $C_{s,d}$. Calculations analogous to those performed in Proposition \ref{pro-Csd} give
\begin{align*}
C_{s,d;p}&=\pi^{\frac{d-1}{2}}\frac{\Gamma\Big(\frac{1+sp}{2}\Big)}{\Gamma\Big(\frac{d+sp}{2}\Big)}\quad 
\hbox{ and }\quad  C_{1,d;p} =\frac{\sigma_d}{2\sqrt{\pi}}\Gamma\left(\frac{d}{2}\right)\frac{\Gamma\left(\frac{p+1}{2}\right)}{\Gamma\left(\frac{p+d}{2}\right)}.
\end{align*}

\begin{theorem}[Gamma-limit at the first scaling]\label{fulldim-p} 
Let $s_0\in[0,1]$ be fixed, and $\{s_\e\}_\e\subset (0,1)$ be such that $s_\e\to s_0$ as $\e\to 0$. 
Let the functional $E_\e$ be defined in $W^{s_\e,p}(\Omega_\e)$ by 
$$E_\e(u)=\frac{1}{\e^{1-s_\e p}}\lfloor u\rfloor_{s_\e}^{p}(\Omega_\e).$$ 
Then the following 
$\Gamma$-convergence results hold with respect to the weak convergence in $L^p_{
\rm loc}(\omega\times (0,1))$ of the corresponding scaled functions:

\smallskip
\indent $\bullet$ if $s_0<1$, then 
$$\Gamma\hbox{\rm -}\lim_{\e\to 0}{E_\e}({u})=C_{s_0,d;p}\int_{\omega}\int_0^1\int_0^1\frac{|v(x^\prime, x_d)-v(x^\prime, y_d)|^p}{|x_d-y_d|^{1+s_0 p}}\, dx_d\, dy_d\, dx^\prime$$
\hspace{8mm} 
with $C_{s_0,d;p}$ as in \eqref{Csdp};

\smallskip\indent $\bullet$ if $s_0=1$, then 
\begin{equation*}\Gamma\hbox{\rm -}\lim_{\e\to 0}(1-s_\e){E_\e}(u)
= 
C_{1,d;p}\int_{\omega}\int_0^1 \Big|\frac{\partial v}{\partial x_d}(x^\prime,x_d)\Big|^p\,dx_d\, dx^\prime.
\end{equation*}
\end{theorem} 

\begin{proof}
For general $p$, we can follow {\em mutatis mutandis} the proof of
Theorem \ref{fulldim}, using the constant 
$C_{s,d;p}$.
\end{proof}
\subsection{Second scaling (dimension reduction)}
The scaling arguments in Section \ref{heu} highlight the critical exponent $s=1/p$. Correspondingly, the energies scale as $\e^2$ for $s<1/p$, as $\e^2|\log\e|$ for $s=1/p$, and as $\e^{1+p-sp}$ for $s>1/p$.
The behaviour for $s\in(0,1)$ is  
summarized in the following theorem. 
The related compactness results can be proved as in the case $p=2$. We remark in particular that for $s<1/p$ the limit exhibits a gain of differentiability of $1/p$.

\begin{theorem}[Gamma-limit at the second scaling] Let $s=s_\e\to s_0\in(0,1)$ as $\e\to 0$. 
Then we have

{\rm(i) (subcritical regime)} if $s_0<1/p$ then we have
$$
\Gamma\hbox{-}\lim_{\e\to 0} \frac{1}{\e^{2}}\lfloor u\rfloor_{s_\e}^p(\Omega_\e)=
\int_\omega\int_\omega 
\frac{|u(x')-u(y')|^p}{|x'-y'|^{d+s_0 p}}dx' dy',
$$
where the $\Gamma$-limit is computed with respect to the dimension-reduction convergence in $L^p$-weak;

{\rm(ii) (critical regime)} for $s=1/p$ we have
$$
\Gamma\hbox{-}\lim_{\e\to 0} \frac{1}{\e^{2}|\log\e|}\lfloor u\rfloor_{1/p}^p(\Omega_\e)=\frac12\int_{S^{d-2}}|\nu'_1|^pd\mathcal H^{d-2}(\nu')\int_\omega |\nabla'u|^pdx'
$$
for all $u\in W^{1,p}(\omega)$, where the $\Gamma$-limit is computed with respect to the dimension-reduction convergence in $L^1_{\rm loc}$;

{\rm(iii) (supercritical regime)} if $s_0>1/p$ then we have
$$
\Gamma\hbox{-}\lim_{\e\to 0} \frac{1}{\e^{1+p-s_\e p}}\lfloor u\rfloor_{s_\e}^p(\Omega_\e)= \frac1{1-s_0} K_{s_0,d;p}\int_\omega |\nabla'u|^pdx'
$$
for all $u\in W^{1,p}$, where the $\Gamma$-limit is computed with respect to the dimension-reduction convergence in $L^1_{\rm loc}$, and
\begin{equation}
K_{s,d;p}=  \frac{2}{p(1+p-sp)} \int_{
 \mathbb R^{d-1}}\frac{|\xi'_1|^p}{(1+|\xi'|^2)^{\frac{d+sp}2}}d\xi'.
\end{equation}
\end{theorem}

\begin{proof} The proof follows closely those of Theorems  
\ref{drconv}, \ref{Gammaths1}, and \ref{Gammaths0}. A minor but repeated change must be made when considering integrals of functions of type $|\langle a,\nu\rangle|^p$ on $S^k$, with $k$ either $d$ or $d-1$. While in the case $p=2$ this quantity is more easily written as $|a|^2\sigma_k/k$, for arbitrary $p$ it can be expressed as 
$$
|a|^p\int_{S^k}|\nu_1|^pd\mathcal H^k.
$$
As an example, when repeating the argument leading to \eqref{corpo},
using the expression above with $a=\nabla'\overline u_\e^{\xi'}$, we obtain the right-hand side 
\begin{eqnarray}\label{corpo-p}
\frac{1}{\e^{1+p-sp}}\int_{B_{K\e}}|\xi'_1|^p C_\e(\xi') \int_{\omega''}
|\nabla'\overline u_\e^{\xi'}|^p dx' d\xi',
\end{eqnarray}
from which we deduce the lower bound for case (ii) and the form of the constant $K_{s,d;p}$.
\end{proof}

Finally, the case $s_\e\to 1$ is described as follows.

\begin{theorem}[Dimension-reduction Gamma-limit for $s\to 1$]\label{Gammath-p} Let $s=s_\e\to 1^-$ as $\e\to 0$. 
Then we have
$$
\Gamma\hbox{-}\lim_{\e\to 0} \frac{1-s}{\e^{1+p-sp}}\lfloor u\rfloor_s^p(\Omega_\e)= \frac12\int_{S^{d-1}}|\nu_1|^pd\mathcal H^{d-1}(\nu)\int_\omega |\nabla'u|^pdx'
$$
for all $u\in H^1(\omega)$, where the $\Gamma$-limit is computed with respect to the dimension-reduction convergence in $L^1_{\rm loc}$.
\end{theorem}

\begin{proof}
 Although the compactness argument still works in this case, the resulting lower bound, corresponding to that in Remark \ref{sub-rem}, is not sharp. Hence, it must be achieved as for the case (iii) in the previous theorem. The rest of the proof is completely analogous to the case $p=2$, except for the expression of the limit constant.
\end{proof}


\section{Concluding remarks}
We have examined the behaviour of Gagliardo seminorms in $H^s$ on thin films of thickness~$\e$. The interplay between thickness and the derivation exponent $s$ results in a variety of limiting behaviour at the dimension-reduction scaling. For low-integrability regimes $ s<1/2$, the relevant interactions are those at finite distance and the energy scaling is $\e^{2}$. After rescaling, the limit energy is still a Gagliardo seminorm of the same form. Due to the reduced dimension of the space, this form highlights an effective gain of $1/2$ in the differentiability exponent. At the critical exponent $s=1/2$ the energy scales as $\e^2|\log\e|$, while for $1/2<s<1$ it scales as $\e^{3-2s}$. The corresponding rescaling acts as if producing a Bourgain-Brezis-Mironescu kernel leading to a dimensionally-reduced Dirichlet integral in the limit. The same holds as $s\to 1$ after scaling by the singular factor $1-s$. 

We finally make some comments in the direction of the derivation of lower-dimensional elasticity theories in the spirit of Le Dret and Raoult \cite{LDR}.  The present analysis can be compared with a recent work on the derivation of variational membrane models in the context of anisotropic nonlocal hyperelasticity in \cite{engl2026derivationvariationalmembranemodels}. There, nonlocal gradients with anisotropic kernels are considered, and the behaviour of the limit functionals depends on suitable dimensionally reduced nonlocal or local gradients according to the scaling of the anisotropic kernels with $\e$. In contrast to our case, the possible behaviour is not determined by $s$. Conversely, for given convolution kernels, which are related to the Bourgain--Brezis--Mironescu approach (see \cite{AABPT}), the dimensionally reduced theories are always local. In that case, the scaling of the energies depends on the ratio between the scaling of the kernel and the thickness of the film (see \cite{Ansinitribuzio}).


\section{Appendix}\label{Appendix}
We gather here the computations related to the hypergeometric integrals defining $C_{s,d}$ and $K_{s,d}$.

\begin{proposition}[Alternative expression for $C_{s,d}$ (Proposition \ref{pro-Csd})]\label{pro-Csd-1}
We have 
\begin{align*}
C_{s,d}=\int_{\mathbb{R}^{d-1}}
\frac{1}{(1+|\xi|^2)^{\frac{d}{2}+s}}
\,d\xi=\frac{2}{\sqrt\pi} \,
\frac{\Gamma(\frac{d}{2}+1)\Gamma(\frac{1}{2}+s)}{\Gamma(\frac{d}{2}+s)}\, \frac{\sigma_d}{2d}=\pi^{\frac{d-1}{2}} \frac{\Gamma(\frac{1}{2}+s)}{\Gamma(\frac{d}{2}+s)}. 
\end{align*}
\end{proposition}

\begin{proof}
In order to carry out the computation, we introduce the quantity
$$J_d^s=\int_0^{+\infty}\frac{z^{d-2}}{(1+z^2)^{\frac{d}{2}+s}}\, dz< +\infty.$$
Integrating in radial coordinates, we obtain
\begin{equation*}
C_{s,d}=\int_{\mathbb R^{d-1}}
\frac{1}{(1+|\xi|^2)^{\frac{d}{2}+s}}\, d\xi = \sigma_{d-1}J_d^s. 
\end{equation*}
For all $d\geq 4$ and $s\in(0,1)$ a simple integration by parts provides the iterative relation 
\begin{equation}
\label{iterj}
J^s_d=\frac{d-3}{d-2+2s}J^s_{d-2}.
\end{equation}
To obtain an explicit formula for $J^s_d$, we then have to compute the products $(d-2+2s)(d-4+2s)
\cdots(2s)$ if $d$ is even, and $(d-2+2s)(d-4+2s)\cdots(1+2s)$ if $d$ is odd. By recursively using the property of the Euler $\Gamma$-function that $\Gamma(1+\alpha)=\alpha\Gamma(\alpha)$, 
it follows that 
\begin{eqnarray*}
&(d-2+2s)(d-4+2s)\cdots(2s)=2^{\frac{d}{2}} \frac{\Gamma(\frac{d}{2}+s)}{\Gamma(s)} & \hbox{\rm if $d$ is even}\\
&(d-2+2s)(d-4+2s)\cdots(1+2s)=2^{\frac{d-1}{2}} \frac{\Gamma(\frac{d}{2}+s)}{\Gamma(\frac{1}{2}+s)}\ \  & \hbox{\rm if $d$ is odd}.
\end{eqnarray*} 
Applying the previous formula with $d-2$ in place of $d$ and $s=\frac{1}{2}$, we obtain
\begin{eqnarray*}
&&(d-3)!! = 2^{\frac{d}{2}-1}\,
\frac{\Gamma\!\left(\frac{d-1}{2}\right)}{\Gamma\!\left(\frac{1}{2}\right)}
\qquad \text{if $d$ is even}, \\[6pt]
&&(d-3)!! = 2^{\frac{d-3}{2}}\,
\frac{\Gamma\!\left(\frac{d-1}{2}\right)}{\Gamma(1)}
\qquad \text{if $d$ is odd}.
\end{eqnarray*}
here $(d-3)!!$ denotes the product of all the positive integers up to $d-3$ that have the same parity (odd or even) as $d-3$.
Hence, recalling that $\Gamma(\frac{1}{2})=\sqrt{\pi}$ and $\Gamma(1)=1$, we obtain the formula
\begin{equation}\label{jds}\nonumber
J_d^s=\begin{cases}
\vspace{2mm}
\displaystyle 
\frac{1}{2\sqrt{\pi}}
\frac{\Gamma(\frac{d-1}{2})
\Gamma(s)}{\Gamma(\frac{d}{2}+s)}\, 2s J^s_2 & \hbox{\rm if $d$ even}\\
\vspace{2mm}
\displaystyle \frac{1}{2}\frac{\Gamma(\frac{d-1}{2})\Gamma(\frac{1}{2}+s)}{\Gamma(\frac{d}{2}+s)}\, (2s+1) J^s_3
& \hbox{\rm if $d$ odd.}
\end{cases}
\end{equation}
It remains to compute 
\begin{eqnarray*}
J^s_2=\int_0^{+\infty} \frac{1}{(1+z^2)^{1+s}}\, dz=\sqrt{\pi}\frac{\Gamma(\frac{1}{2}+s)}{2\Gamma(1+s)}, \quad J^s_3=\int_0^{+\infty} \frac{z}{(1+z^2)^{1+s}}\, dz=\frac{1}{1+2s}, 
\end{eqnarray*}
which give the explicit formula for $J_d^s$
\begin{equation}\label{formulaj}
J_d^s=
\frac{1}{2}
\frac{\Gamma(\frac{d-1}{2})\Gamma(\frac{1}{2}+s)}{\Gamma(\frac{d}{2}+s)} 
\end{equation}
 in terms of the function $\Gamma$.
 Recalling that $\sigma_{d-1}=2\frac{\pi^{\frac{d-1}{2}}}{\Gamma(\frac{d-1}{2})}$ and using \eqref{formulaj} we can write
\[
C_{s,d}=\sigma_{d-1}J_d^s=\pi^{\frac{d-1}{2}} \frac{\Gamma(\frac{1}{2}+s)}{\Gamma(\frac{d}{2}+s)}.
\]
and,  using the relation $\frac{\sigma_d}{\sigma_{d-1}}
=\frac{d}{d-1}\ \frac{\Gamma(\frac{d+1}{2})}{\Gamma(\frac{d}{2}+1)}\, \sqrt\pi$,
we also have 
\[
C_{s,d}= \frac{2}{\sqrt\pi} \,
\frac{\Gamma(\frac{d}{2}+1)\Gamma(\frac{1}{2}+s)}{\Gamma(\frac{d}{2}+s)}\, \frac{\sigma_d}{2d}, 
\]
and the proof is completed.
\end{proof}

\begin{remark}[Relation between $K_{s,d}$ and $C_{s,d}$ (Remark \ref{Ksd})]\label{Ksd-1}\rm 
The coefficients $C_{s,d}$ and $K_{s,d}$ are related by $$
K_{s,d}=\frac{1}{(2s-1)(3-2s)}C_{s,d}.
$$
We write 
$$K_{s,d}=\frac{\sigma_{d-1}}{(3-2s)(d-1)}I_s^d,
\hbox{ where }
I_d^s=\int_{0}^{+\infty}\frac{z^d}{(1+z^2)^{\frac{d}{2}+s}}\, dz.$$ 
Since an iterative formula corresponding to \eqref{iterj} also holds for $I_s^d$ with $d-1$ in the place of $d-3$; that is, 
$$I^s_d=\frac{d-1}{d-2+2s}I^s_{d-2}$$
for all $d\geq 2$ and $s\in (\frac{1}{2},1),$ following the computations in Proposition \ref{pro-Csd-1}, we obtain 
$$I_s^d=
\begin{cases}
\vspace{2mm}
\displaystyle 
\frac{d-1}{2\sqrt{\pi}}
\frac{\Gamma(\frac{d-1}{2})
\Gamma(s)}{\Gamma(s+\frac{d}{2})}\, I^s_0 = \frac{d-1}{4}
\frac{\Gamma(\frac{d-1}{2}))
}{\Gamma(s+\frac{d}{2})}\, \Gamma(s-\frac{1}{2})& \hbox{\rm if $d$ even}\\
\vspace{2mm}
\displaystyle \frac{d-1}{2}\frac{\Gamma(\frac{d-1}{2})\Gamma(s+\frac{1}{2})}{\Gamma(s+\frac{d}{2})}\, I^s_1=\frac{d-1}{4}\frac{\Gamma(\frac{d-1}{2})}{\Gamma(s+\frac{d}{2})}\, \frac{\Gamma(s+\frac{1}{2})}{s-\frac{1}{2}}
& \hbox{\rm if $d$ odd.}
\end{cases}$$
Noting that $(s-\frac{1}{2})\Gamma(s-\frac{1}{2})=\Gamma(s+\frac{1}{2})$, we have the following representation formula for $K_{s,d}$ in terms of the Gamma function 
$$K_{s,d}=\frac{\sigma_{d-1}}{2(2s-1)(3-2s)}\frac{\Gamma(\frac{d-1}{2})\Gamma(s+\frac{1}{2})}{\Gamma(s+\frac{d}{2})}.$$
Recalling the expression of the constant $C_{s,d}$ in \eqref{csd-1}, 
the desired relation 
follows for all $d\geq 2$ and $s\in(\frac{1}{2},1)$. 
 \end{remark}

\begin{remark}\rm We have
$$
\lim_{s\to 1^-}K_{s,d}=\frac{\sigma_d}{2d}.
$$

We can check this directly from the definition of $K_{s,d}$, without using its rewriting in Remark \ref{Ksd-1}. We write
$$
K_{s,d}=\frac{\sigma_{d-1}}{(3-2s)(d-1)}\int_0^{+\infty}\frac{z^d}{(1+z^2)^{\frac d2+s}}dz=\frac{\sigma_{d-1}}{(3-2s)(d-1)}I_d^1,
$$
and note that
\begin{equation}\label{ks1}
\lim_{s\to 1^-}K_{s,d}=\frac{\sigma_{d-1}}{d-1}I^1_d. 
\end{equation}
The general recursive formula obtained by integration by parts 
in the case $s=1$ simply reduces to 
$I^1_d=\frac{d-1}{d}I_{d-2}$,
so that 
$$I^1_d=\begin{cases}
\vspace{2mm}
\displaystyle\frac{(d-1)!!}{d!!}I^1_0=\frac{(d-1)!!}{d!!}\frac{\pi}{2} & \hbox{\rm if $d$ even}\\
\vspace{2mm}
\displaystyle\frac{(d-1)!!}{d!!}I^1_1=\frac{(d-1)!!}{d!!} & \hbox{\rm if $d$ odd.}
\end{cases}
$$
Recalling the expression 
$$\frac{\sigma_d}{\sigma_{d-1}}
=\frac{d}{d-1}\frac{\Gamma(\frac{d+1}{2})}{\Gamma(\frac{d}{2}+1)} \sqrt\pi=
\begin{cases}
\vspace{2mm}
\displaystyle \frac{d}{d-1}
\frac{(d-1)!!}{d!!}\pi
& \hbox{\rm if $d$ even}\\
\vspace{2mm}
\displaystyle \frac{2d}{d-1} 
\frac{(d-1)!!}{d!!}
& \hbox{\rm if $d$ odd,}
\end{cases}
$$
used in the proof of Proposition \ref{pro-Csd},
we then have 
$I^1_d=\frac{d-1}{2d}\frac{\sigma_d}{\sigma_{d-1}}$
for all $d\geq 2$.
Hence, recalling \eqref{ks1}, we complete the proof.
\end{remark}

\bibliographystyle{abbrv}
\bibliography{references}

\begin{thebibliography}{10}

\bibitem{MR4165063}
A.~Alberico, A.~Cianchi, L.~s. Pick, and L.~Slav\'ikov\'a.
\newblock On the limit as {$s\to 0^+$} of fractional {O}rlicz-{S}obolev spaces.
\newblock {\em J. Fourier Anal. Appl.}, 26(6):Paper No. 80, 19, 2020.

\bibitem{AABPT}
R.~Alicandro, N.~Ansini, A.~Braides, A.~Piatnitski, and A.~Tribuzio.
\newblock {\em A {V}ariational {T}heory of {C}onvolution-type {F}unctionals}.
\newblock SpringerBriefs on PDEs and Data Science. Springer-Verlag, Berlin,
  2023.

\bibitem{Ansinitribuzio}
N.~Ansini and A.~Tribuzio.
\newblock Multiscale analysis and homogenization of nonlocal thin films.
  {P}reprint, 2026.

\bibitem{AnzB}
G.~Anzellotti and S.~Baldo.
\newblock Asymptotic development by {$\Gamma$}-convergence.
\newblock {\em Appl. Math. Optim.}, 27(2):105--123, 1993.

\bibitem{BBM}
J.~Bourgain, H.~Brezis, and P.~Mironescu.
\newblock Another look at {S}obolev spaces.
\newblock In {\em Optimal {C}ontrol and {P}artial {D}ifferential {E}quations},
  pages 439--455. IOS, Amsterdam, 2001.

\bibitem{GCB}
A.~Braides.
\newblock {\em {$\Gamma$}-{C}onvergence for {B}eginners}, volume~22 of {\em
  Oxford Lecture Series in Mathematics and its Applications}.
\newblock Oxford University Press, Oxford, 2002.

\bibitem{Handbook}
A.~Braides.
\newblock {A} handbook of {$\Gamma$}-convergence.
\newblock In M.~Chipot and P.~Quittner, editors, {\em Handbook of Differential
  Equations: Stationary Partial Differential Equations}, volume~3, pages
  101--213. North-Holland, 2006.

\bibitem{BBD}
A.~Braides, G.~C. Brusca, and D.~Donati.
\newblock Another look at elliptic homogenization.
\newblock {\em Milan J. Math.}, 92(1):1--23, 2024.

\bibitem{BFF}
A.~Braides, I.~Fonseca, and G.~Francfort.
\newblock 3{D}-2{D} asymptotic analysis for inhomogeneous thin films.
\newblock {\em Indiana Univ. Math. J.}, 49(4):1367--1404, 2000.

\bibitem{BT}
A.~Braides and L.~Truskinovsky.
\newblock Asymptotic expansions by {$\Gamma$}-convergence.
\newblock {\em Contin. Mech. Thermodyn.}, 20(1):21--62, 2008.

\bibitem{MR4525722}
D.~Brazke, A.~Schikorra, and P.-L. Yung.
\newblock Bourgain-{B}rezis-{M}ironescu convergence via {T}riebel-{L}izorkin
  spaces.
\newblock {\em Calc. Var. Partial Differential Equations}, 62(2):Paper No. 41,
  33, 2023.

\bibitem{MR3556344}
H.~Brezis and H.-M. Nguyen.
\newblock The {BBM} formula revisited.
\newblock {\em Atti Accad. Naz. Lincei Rend. Lincei Mat. Appl.},
  27(4):515--533, 2016.

\bibitem{MR4275122}
H.~Brezis, J.~Van~Schaftingen, and P.-L. Yung.
\newblock A surprising formula for {S}obolev norms.
\newblock {\em Proc. Natl. Acad. Sci. USA}, 118(8):Paper No. e2025254118, 6,
  2021.

\bibitem{MR1942116}
K.~Brezis.
\newblock How to recognize constant functions. {A} connection with {S}obolev
  spaces.
\newblock {\em Uspekhi Mat. Nauk}, 57(4(346)):59--74, 2002.

\bibitem{MR4453966}
F.~Buseghin, N.~Garofalo, and G.~Tralli.
\newblock On the limiting behaviour of some nonlocal seminorms: a new
  phenomenon.
\newblock {\em Ann. Sc. Norm. Super. Pisa Cl. Sci. (5)}, 23(2):837--875, 2022.

\bibitem{DM}
G.~Dal~Maso.
\newblock {\em An Introduction to $\Gamma$-convergence}.
\newblock Birkh\"auser, Boston, 1993.

\bibitem{MR1942130}
J.~D\'avila.
\newblock On an open question about functions of bounded variation.
\newblock {\em Calc. Var. Partial Differential Equations}, 15(4):519--527,
  2002.

\bibitem{pin25}
E.~Davoli, G.~D. Fratta, R.~Giorgio, and A.~Pinamonti.
\newblock Necessary and sufficient conditions for the {M}azya--{S}haposhnikova
  formula in (fractional) {S}obolev space. {P}reprint, 2025.

\bibitem{Davo}
E.~Davoli, G.~D. Fratta, and V.~Pagliari.
\newblock Sharp conditions for the validity of the
  {B}ourgain–{B}rezis–{M}ironescu formula.
\newblock {\em Proc. R. Soc. Edinb., Sect. A, Math.}, 56(1):69--92, 2026.

\bibitem{DPV}
E.~{Di~Nezza}, G.~Palatucci, and E.~Valdinoci.
\newblock Hitchhiker's guide to the fractional {S}obolev spaces.
\newblock {\em Bull. Sci. Math.}, 136:521--573, 2012.

\bibitem{engl2026derivationvariationalmembranemodels}
D.~Engl, A.~Molchanova, and H.~Schönberger.
\newblock Derivation of variational membrane models in the context of
  anisotropic nonlocal hyperelasticity, 2026, arXiv2602.17278.

\bibitem{MR1916989}
G.~Friesecke, R.~D. James, and S.~M\"uller.
\newblock A theorem on geometric rigidity and the derivation of nonlinear plate
  theory from three-dimensional elasticity.
\newblock {\em Comm. Pure Appl. Math.}, 55(11):1461--1506, 2002.

\bibitem{MR2210909}
G.~Friesecke, R.~D. James, and S.~M\"uller.
\newblock A hierarchy of plate models derived from nonlinear elasticity by
  gamma-convergence.
\newblock {\em Arch. Ration. Mech. Anal.}, 180(2):183--236, 2006.

\bibitem{MR5012373}
L.~Gennaioli and G.~Stefani.
\newblock Sharp conditions for the {BBM} formula and asymptotics of heat
  content-type energies.
\newblock {\em Arch. Ration. Mech. Anal.}, 250(1):Paper No. 8, 46, 2026.

\bibitem{MR4953682}
B.-X. Han, A.~Pinamonti, Z.~Xu, and K.~Zambanini.
\newblock Maz'ya-{S}haposhnikova meet {B}ishop-{G}romov.
\newblock {\em Potential Anal.}, 63(2):513--529, 2025.

\bibitem{MR4788002}
P.~Lahti, A.~Pinamonti, and X.~Zhou.
\newblock B{V} functions and nonlocal functionals in metric measure spaces.
\newblock {\em J. Geom. Anal.}, 34(10):Paper No. 318, 34, 2024.

\bibitem{MR4682567}
P.~Lahti, A.~Pinamonti, and X.~Zhou.
\newblock A characterization of {BV} and {S}obolev functions via nonlocal
  functionals in metric spaces.
\newblock {\em Nonlinear Anal.}, 241:Paper No. 113467, 14, 2024.

\bibitem{LDR}
H.~Le~Dret and A.~Raoult.
\newblock The nonlinear membrane model as variational limit of nonlinear
  three-dimensional elasticity.
\newblock {\em J. Math. Pures Appl. (9)}, 74(6):549--578, 1995.

\bibitem{leofrac}
G.~Leoni.
\newblock {\em A {F}irst {C}ourse in {F}ractional {S}obolev {S}paces}, volume
  229 of {\em Graduate Studies in Mathematics}.
\newblock American Mathematical Society, Providence, RI, 2023.

\bibitem{MR2832587}
G.~Leoni and D.~Spector.
\newblock Characterization of {S}obolev and {$BV$} spaces.
\newblock {\em J. Funct. Anal.}, 261(10):2926--2958, 2011.

\bibitem{MR3132740}
G.~Leoni and D.~Spector.
\newblock Corrigendum to ``{C}haracterization of {S}obolev and {$BV$} spaces''
  [{J}. {F}unct. {A}nal. 261 (10) (2011) 2926--2958].
\newblock {\em J. Funct. Anal.}, 266(2):1106--1114, 2014.

\bibitem{MR4374610}
A.~Maione, A.~M. Salort, and E.~Vecchi.
\newblock Maz'ya-{S}haposhnikova formula in magnetic fractional
  {O}rlicz-{S}obolev spaces.
\newblock {\em Asymptot. Anal.}, 126(3-4):201--214, 2022.

\bibitem{MS}
V.~Maz'ya and T.~Shaposhnikova.
\newblock On the {B}ourgain, {B}rezis, and {M}ironescu theorem concerning
  limiting embeddings of fractional {S}obolev spaces.
\newblock {\em J. Funct. Anal.}, 195(2):230--238, 2002.

\bibitem{MR4999552}
H.-M. Nguyen.
\newblock Characterizations of the {S}obolev norms and the total variation via
  nonlocal functionals, and related problems.
\newblock {\em C. R. Math. Acad. Sci. Paris}, 363:1429--1455, 2025.

\bibitem{Picerni}
M.~Picerni.
\newblock Analysis for non-local phase transitions close to the critical
  exponent {$s=\frac{1}{2}$}.
\newblock {\em Ric. Mat.}, 74(3):1599--1626, 2025.

\bibitem{ponce}
A.~C. Ponce.
\newblock A new approach to {S}obolev spaces and connections to
  {$\Gamma$}-con\-ver\-gence.
\newblock {\em Calc. Var. Partial Differential Equations}, 19(3):229--255,
  2004.

\bibitem{SV}
O.~Savin and E.~Valdinoci.
\newblock {$\Gamma$}-convergence for nonlocal phase transitions.
\newblock {\em Ann. Inst. H. Poincar\'e{} C Anal. Non Lin\'eaire},
  29(4):479--500, 2012.

\bibitem{Solci-vortices}
M.~Solci.
\newblock Nonlocal-interaction vortices.
\newblock {\em SIAM J. Math. Anal.}, 56(3):3430--3451, 2024.

\bibitem{Solci}
M.~Solci.
\newblock Higher-order non-local gradient theory of phase-transitions.
\newblock {\em Milan J. Math.}, 93:455--486, 2025.

\end{thebibliography}

\end{document}